\documentclass[11pt]{article}

\usepackage[margin=1in]{geometry}
\usepackage[T1]{fontenc}
\usepackage[utf8]{inputenc}
\usepackage{lmodern}
\usepackage{microtype}
\usepackage{amsmath,amssymb,amsthm,mathtools}
\usepackage{booktabs,array}
\usepackage{enumitem}
\usepackage{xcolor}
\usepackage[hidelinks]{hyperref}
\usepackage{doi}
\usepackage{authblk}
\usepackage[backend=biber,citestyle=numeric-comp,bibstyle=ieee,sorting=none,minbibnames=5,maxbibnames=8,giveninits=true]{biblatex}
\AtBeginBibliography{\footnotesize\setlength{\bibitemsep}{0.05\baselineskip}}
\hypersetup{
  pdftitle={Congruence Jumping in Polynomial Divisibility Systems: Arithmetic Structure and Solution Sets},
  pdfauthor={Max A. Alekseyev and Dmitry I. Khomovsky},
  pdfsubject={Polynomial divisibility systems, Diophantine congruences, arithmetic lifting, and Pell dynamics},
  pdfkeywords={congruence jumping, polynomial congruences, Diophantine equations, Vieta jumping, reciprocal polynomials, Pell equations}
}
\newtheorem{theorem}{Theorem}[section]
\newtheorem{proposition}[theorem]{Proposition}
\newtheorem{lemma}[theorem]{Lemma}
\newtheorem{corollary}[theorem]{Corollary}

\newtheorem{example}[theorem]{Example}
\theoremstyle{definition}
\newtheorem{definition}[theorem]{Definition}
\theoremstyle{remark}
\newtheorem{remark}[theorem]{Remark}

\newcommand{\defeq}{:=}
\newcommand{\Z}{\mathbb Z}
\newcommand{\Q}{\mathbb Q}
\newcommand{\cS}{\mathcal S}

\newcommand{\Rcal}{\mathcal R}

\newcommand{\Lcal}{\mathcal L}

\DeclareMathOperator{\cont}{cont}
\title{Congruence Jumping in Polynomial Divisibility Systems:\\Arithmetic Structure and Solution Sets}
\author{Max A. Alekseyev\thanks{The George Washington University, Washington, DC, USA. Email: \href{mailto:maxal@gwu.edu}{maxal@gwu.edu}} \and Dmitry I. Khomovsky\thanks{Independent researcher. Email: \href{mailto:dmkhomovsky@gmail.com}{dmkhomovsky@gmail.com}}}
\date{}
\begin{document}
\maketitle

\begin{abstract}
Let $A$ and $B$ be polynomials with integer coefficients.  We study nonzero integer solutions of
\[
  y\mid A(x),\qquad x\mid B(y),
\]
using quotient transformations that may change the polynomial pair.  We call this process \emph{congruence jumping}.  We establish a general reciprocal-polynomial rule governing such jumps, without coprimality, monicity, or unit assumptions, and give exact criteria for when the resulting polynomial states can be normalized integrally.  Companion-surface identities provide a mechanism for constructing infinite integral quotient chains, including an explicit mixed-degree example.

The quadratic case is considerably more rigid.  We obtain a quantitative denominator bound for finite chains, classify the exceptional nonconstant one-sided infinite chains with nonintegral conic parameter, and show that changing quadratic states reduce to ordinary Vieta dynamics on a fixed conic.  We further construct a genuinely nonunit recurrent four-cycle with infinitely many positive integral points and show that its dynamics admits a uniform Pell-type linearization.  A motivating nonunit divisibility system is analyzed through Pell orbits and changing-state ladders.  Finally, an independent relation-lattice criterion reduces certain solution sets to a finite divisor search.
\end{abstract}



\section{Introduction}

Vieta jumping replaces one root of a quadratic Diophantine equation by the other root and therefore stays on a fixed equation.  This point of view goes back at least to the quotient and root-replacement chains of Mills~\cite{Mills1953,Mills1954,Mills1956}; modern systematic treatments include Grechuk~\cite{Grechuk2024} and the recent Vieta--Pell perspective of Lemmermeyer~\cite{Lemmermeyer2026}.

We study the related divisibility system
\begin{equation}\label{eq:basic-div}
  y\mid A(x),\qquad x\mid B(y),
\end{equation}
for integer polynomials $A,B$.  Our ultimate aim is to find or describe its integer solution set
\[
  \cS(A,B)=\{(x,y)\in(\Z\setminus\{0\})^2:y\mid A(x),\ x\mid B(y)\}.
\]
As in classical Vieta jumping, a satisfactory description may take the form of descent to finitely many terminal points, explicit recurrence or Pell-type parametrizations of solution components, or finite arithmetic criteria that exhaust all possibilities.

If $y'=A(x)/y$, then $y'$ is integral.  The second divisibility generally does not retain the same polynomial $B$; instead it transforms by a reciprocal-polynomial rule.  We call iteration of this operation \emph{congruence jumping}.  The pointwise arithmetic data are transported through changing polynomial states, and only after several jumps may a state return.  Finite closure and normalization are therefore intermediate questions: their value lies in converting quotient dynamics into information about the original solution set.

The immediate motivation for this viewpoint came from a 2024 MathOverflow discussion initiated by Bogdan Grechuk.  In a comment, Denis Shatrov observed that for the pair of divisibility problems
\[
  xy\mid x^3+x^d+1+y^2,\qquad d\in\{1,2\},
\]
quotient replacements can produce new solutions while switching the polynomial state from $d$ to $3-d$; Grechuk subsequently used these transformations to solve the original problem~\cite{GrechukShatrov2024}.  This is a concrete precursor of the phenomenon studied here: a quotient remains integral while the defining polynomial state changes.

Alternating quotient chains for polynomial divisibility systems are classical.  Mills treated self-reciprocal quadratic families~\cite{Mills1953,Mills1954,Mills1956}, Mohanty the cubic system $x\mid y^3+1$, $y\mid x^3+1$~\cite{Mohanty1977}, and Dofs mixed-degree extensions with several solution-orbit components~\cite{Dofs1993}.  Bier studies the general ambient system $x\mid g(y)$, $y\mid f(x)$ under positivity hypotheses and constructs infinite quotient chains for broad degree ranges~\cite{Bier2014}.  Related polynomial congruences occur in work of Jacobsthal, Mordell, Schinzel, and Tengely--Ulas~\cite{Jacobsthal1939,Mordell1952,Schinzel2015,Schinzel2018,TengelyUlas2022}.

A complementary recent approach of Grechuk and Agbanwa studies when values of a polynomial $P(x)$ are represented infinitely often by a fixed nondegenerate integral binary quadratic form.  For certain cubic and quartic $P$, and for certain compositions $P=R\circ Q$ with $\deg R=3$ and $\deg Q=2$, their tangent construction reduces the representation problem to an auxiliary quadratic equation together with explicit congruence conditions~\cite{GrechukAgbanwa2026}.  This is close in arithmetic spirit to our use of auxiliary quadratic and Pell dynamics, but complementary in mechanism: their quadratic form remains fixed, whereas congruence jumping transports mutual polynomial divisibilities through changing reciprocal-polynomial states.

There is also a global geometric counterpart.  Koll\'ar--Villalobos-Paz construct a four-surface companion groupoid and an infinite-order automorphism by composing quotient maps, with reciprocal-type coefficient changes in their normalized companion construction~\cite{KollarVillalobosPaz2025}; Koll\'ar--Li use this machinery in the cubic Mordell--Schinzel problem~\cite{KollarLi2024}.  Our emphasis is pointwise and arithmetic: we keep the raw reciprocal transform integral, identify precisely when primitive normalization is valid at the current point, and ask what this implies about all integer solutions.

The contribution is arithmetic rather than merely coefficient-dynamical.  The main results can be summarized as follows.
\begin{enumerate}[label=\textup{(\roman*)}]
\item Theorem~\ref{thm:local-reciprocal} transports a solution through a raw reciprocal transform without coprimality, monicity, or endpoint-unit hypotheses.  Proposition~\ref{prop:content-obstruction} gives the exact local obstruction to primitive normalization.  Proposition~\ref{prop:companion-surface-propagation} and Corollaries~\ref{cor:companion-indefinite}--\ref{cor:common-constant-indefinite} give a broad indefinite lifting mechanism.  In particular, Proposition~\ref{prop:c5-infinite-lift} gives an explicit mixed-degree $(5,3)$ surface on which every nonzero integral point starts an infinite integral quotient chain.

\item In degree two, Proposition~\ref{prop:quadratic-denominator-rigidity} gives a finite quantitative rigidity theorem: if an integral chain has $L$ terms and conic parameter $H=a/b$ in lowest terms, then $b^{L-4}$ divides a fixed Cassini-type determinant.  Corollary~\ref{cor:quadratic-infinite-rigidity} classifies the exceptional nonconstant one-sided infinite chains with nonintegral $H$, while Proposition~\ref{prop:quadratic-gauge} reduces changing quadratic states rationally to ordinary Vieta jumping on a fixed conic.

\item Proposition~\ref{prop:recurrent-nonunit-cycle} constructs a recurrent genuinely nonunit four-cycle with infinitely many positive integral points.  Proposition~\ref{prop:lucas-pell-cycle} identifies its uniform Lucas--Pell structure.  Section~\ref{sec:nonunit} then treats the motivating system
\[
  m\mid 2n^2-2n+1,\qquad n\mid 2m^2-2m+1,
\]
separating fixed-parameter Pell infinitude from changing-state arithmetic lifting and giving exact parity and one-step return criteria.

\item As a complementary finite method, Section~\ref{sec:relation-lattices} records a full-rank auxiliary-relation certificate: if a fixed relation matrix $V$ is valid for every solution, then $XY\mid\det V$, so the nonzero solution set is exhausted by a finite divisor search.  This criterion is independent of the existence of a return orbit.
\end{enumerate}

The paper is organized accordingly.  Section~\ref{sec:states} develops reciprocal transport, normalization, companion-surface lifting, the short rational gauge closure, and higher-degree arithmetic examples.  Section~\ref{sec:four-jump} treats finite reciprocal cycles.  Section~\ref{sec:quadratic} develops quadratic rigidity, fixed-conic reduction, recurrent nonunit states, and Lucas--Pell dynamics.  Section~\ref{sec:nonunit} gives the detailed nonunit case study.  Section~\ref{sec:relation-lattices} gives the independent finite-classification certificate, and Section~\ref{sec:applications} collects selected descents and classifications.  We deliberately leave the intrinsic integral classification of coefficient orbits to a separate problem; it is not needed for the solution-set results proved here.

\section{Reciprocal transport and arithmetic lifting}\label{sec:states}

Let $P(T)=\sum_{j=0}^{d}p_jT^j\in\Z[T]$ have degree $d$.  For $c\in\Z$ define its \emph{$c$-reciprocal transform} by
\begin{equation}\label{eq:c-reciprocal}
  \Rcal_cP(T)\defeq T^dP(c/T)
  =\sum_{j=0}^{d}p_jc^jT^{d-j}.
\end{equation}
The right-hand expression makes clear that $\Rcal_cP\in\Z[T]$ even when $c$ is not a unit.  For $c=1$ this is the ordinary reciprocal polynomial
\begin{equation}\label{eq:reciprocal-def}
  P^{\vee}(T)\defeq\Rcal_1P(T)=T^dP(T^{-1}).
\end{equation}
For $c=-1$ the same construction gives the signed reciprocal $T^dP(-T^{-1})$.

\begin{definition}\label{def:state}
For $A,B\in\Z[T]$, define the \emph{divisibility state}
\[
  \cS(A,B)\defeq
  \{(x,y)\in(\Z\setminus\{0\})^2:\ y\mid A(x),\ x\mid B(y)\}.
\]
A point is \emph{nondegenerate} for a displayed sequence of jumps if none of the quotient coordinates occurring in that sequence is zero.
\end{definition}

The purpose of transporting solutions through changing states is ultimately to constrain $\cS(A,B)$ itself.  A closed integral state orbit can support descent or recurrence arguments, while a normalization obstruction can exclude proposed solution trajectories; in favorable cases these mechanisms lead to complete classifications rather than merely to new trajectories.

The following observation is the basic arithmetic operation of the paper.

\subsection{Reciprocal transport and primitive normalization}\label{subsec:local-transport}

\begin{theorem}[local reciprocal transformation]\label{thm:local-reciprocal}
Let $A,B\in\Z[T]$ have degrees $n,m$, respectively, and let $(x,y)\in\cS(A,B)$.  Write
\[
  a_0\defeq A(0),\qquad b_0\defeq B(0).
\]
If
\begin{equation}\label{eq:local-yprime}
  y'\defeq\frac{A(x)}{y}\ne0,
\end{equation}
then
\begin{equation}\label{eq:local-y-state}
  (x,y')\in\cS\bigl(A,\Rcal_{a_0}B\bigr).
\end{equation}
Symmetrically, if
\begin{equation}\label{eq:local-xprime}
  x'\defeq\frac{B(y)}{x}\ne0,
\end{equation}
then
\begin{equation}\label{eq:local-x-state}
  (x',y)\in\cS\bigl(\Rcal_{b_0}A,B\bigr).
\end{equation}
No coprimality, monicity, or unit hypothesis on the endpoint coefficients is required.
\end{theorem}

\begin{proof}
Write $B(T)=\sum_{j=0}^{m}b_jT^j$.  From $yy'=A(x)$ we have
\[
  yy'\equiv a_0\pmod x.
\]
Since $x\mid B(y)$, multiplication by $(y')^m$ gives
\begin{align*}
  0
  &\equiv (y')^mB(y)\\
  &=\sum_{j=0}^{m}b_jy^j(y')^m\\
  &\equiv\sum_{j=0}^{m}b_ja_0^j(y')^{m-j}
   =\Rcal_{a_0}B(y')\pmod x.
\end{align*}
Thus $x\mid\Rcal_{a_0}B(y')$, while $y'\mid A(x)$ follows from \eqref{eq:local-yprime}.  This proves \eqref{eq:local-y-state}; the second assertion is symmetric.
\end{proof}

\begin{remark}\label{rem:no-inversion}
The proof deliberately does not divide by $a_0$ modulo $x$.  Even when $a_0$ is not invertible modulo $x$, the product congruence $yy'\equiv a_0\pmod x$ is enough after multiplication by $(y')^m$.  This is the sense in which Theorem~\ref{thm:local-reciprocal} is more local and less restrictive than a global companion-surface isomorphism.
\end{remark}

\begin{proposition}[content and the normalization obstruction]\label{prop:content-obstruction}
Let
\[
  B(T)=\sum_{j=0}^{m}b_jT^j\in\Z[T]
\]
be primitive with $b_0\ne0$, let $c\in\Z\setminus\{0\}$, and write
\begin{equation}\label{eq:primitive-reciprocal}
  \Rcal_cB=gC,
  \qquad
  g\defeq\cont(\Rcal_cB)>0,
\end{equation}
with $C\in\Z[T]$ primitive.  With the convention $v_p(0)=+\infty$, one has
\begin{equation}\label{eq:content-valuation}
  v_p(g)
  =\min_{0\le j\le m}
   \bigl(v_p(b_j)+jv_p(c)\bigr)
\end{equation}
for every prime $p$.  In particular,
\begin{equation}\label{eq:content-support}
  p\mid g
  \quad\Longleftrightarrow\quad
  p\mid c\ \text{and}\ p\mid B(0).
\end{equation}

Suppose now that $(x,y)\in\cS(A,B)$, $A(0)=c$, and
\[
  y'\defeq\frac{A(x)}y\ne0.
\]
Then the raw jump of Theorem~\ref{thm:local-reciprocal} can be normalized from $gC$ to $C$ at the current point precisely when
\begin{equation}\label{eq:exact-normalization-criterion}
  v_p(C(y'))\ge v_p(x)
  \qquad\text{for every prime }p\mid g.
\end{equation}
Thus $\gcd(x,g)=1$ is sufficient, but not necessary, for primitive normalization.
\end{proposition}

\begin{proof}
The coefficients of $\Rcal_cB$ are $b_jc^j$, in reverse order, so taking the minimum of their $p$-adic valuations gives \eqref{eq:content-valuation}.  If $p\mid c$ and $p\mid b_0$, then every coefficient of $\Rcal_cB$ is divisible by $p$.  Conversely, if $p\mid g$, then $p\mid b_0$.  If $p\nmid c$, \eqref{eq:content-valuation} would force every $b_j$ to be divisible by $p$, contradicting primitivity.  This proves \eqref{eq:content-support}.

Theorem~\ref{thm:local-reciprocal} gives $x\mid gC(y')$.  For $p\nmid g$ this already implies $v_p(C(y'))\ge v_p(x)$.  Hence $x\mid C(y')$ is equivalent exactly to the remaining inequalities \eqref{eq:exact-normalization-criterion} at primes dividing $g$.
\end{proof}

\begin{corollary}[common-constant surface normalization]\label{cor:surface-normalization}
Let $A,B\in\Z[T]$ be polynomials of positive degree with
\[
  A(0)=B(0)=c\ne0,
\]
and let $x,y,z\in\Z$ with $xy\ne0$ satisfy
\begin{equation}\label{eq:common-constant-surface}
  A(x)+B(y)-c=zxy.
\end{equation}
Then $(x,y)\in\cS(A,B)$.  Moreover, define
\[
  C_A(T)\defeq\frac{1}{c}\Rcal_cA(T),
  \qquad
  C_B(T)\defeq\frac{1}{c}\Rcal_cB(T).
\]
Both $C_A,C_B$ lie in $\Z[T]$ and are primitive.  If
\[
  y'\defeq\frac{A(x)}y\ne0,
\]
then
\[
  (x,y')\in\cS(A,C_B).
\]
Symmetrically, if
\[
  x'\defeq\frac{B(y)}x\ne0,
\]
then
\[
  (x',y)\in\cS(C_A,B).
\]
Thus the surface relation itself validates cancellation of the full content $|c|$ from the corresponding raw reciprocal transform.
\end{corollary}

\begin{proof}
Reducing \eqref{eq:common-constant-surface} modulo $y$ gives $y\mid A(x)$, while reduction modulo $x$ gives $x\mid B(y)$.  Hence $(x,y)\in\cS(A,B)$.

Write
\[
  B(T)=c+\sum_{j=1}^{m}b_jT^j.
\]
Then
\[
  C_B(T)=T^m+\sum_{j=1}^{m}b_jc^{j-1}T^{m-j},
\]
so $C_B\in\Z[T]$ and is primitive; the same argument applies to $C_A$.  Substituting $A(x)=yy'$ into \eqref{eq:common-constant-surface} and dividing by $y$ gives
\[
  y'+\sum_{j=1}^{m}b_jy^{j-1}=zx.
\]
Multiplying by $(y')^{m-1}$ and reducing modulo $x$, we use
\[
  yy'=A(x)\equiv c\pmod x
\]
to obtain
\begin{align*}
  0
  &\equiv (y')^m+\sum_{j=1}^{m}b_j(yy')^{j-1}(y')^{m-j}\\
  &\equiv (y')^m+\sum_{j=1}^{m}b_jc^{j-1}(y')^{m-j}
   =C_B(y')\pmod x.
\end{align*}
Thus $x\mid C_B(y')$, while $y'\mid A(x)$ by definition, proving $(x,y')\in\cS(A,C_B)$.  The statement for $x'$ is symmetric.
\end{proof}

\subsection{Companion surfaces and indefinite lifting}\label{subsec:companion-surfaces}

\begin{proposition}[companion-surface propagation]\label{prop:companion-surface-propagation}
Let $P,Q\in\Z[T]$ be primitive, let $m\defeq\deg Q\ge2$, and write
\[
  a\defeq P(0),\qquad b\defeq Q(0),\qquad
  \kappa\defeq\frac{a}{b}\in\Z\setminus\{0\}.
\]
Define
\begin{equation}\label{eq:companion-C-def}
  C(T)\defeq\frac1b\Rcal_aQ(T),
  \qquad
  c\defeq C(0),
  \qquad
  \mu\defeq\frac ca.
\end{equation}
Then $C\in\Z[T]$ is monic and primitive, and
\begin{equation}\label{eq:companion-multiplier-update}
  \mu
  =\operatorname{lc}(Q)\,\kappa^{m-1}b^{m-2}
  \in\Z\setminus\{0\}.
\end{equation}
Put
\[
  H_Q(T)\defeq\frac{Q(T)-b}{T},
  \qquad
  H_C(T)\defeq\frac{C(T)-c}{T}.
\]
Then one has the polynomial divisibility
\begin{equation}\label{eq:companion-polynomial-identity}
  Q(T)\mid
  \mu T+H_C\!\left(-\kappa H_Q(T)\right)
  \qquad\text{in }\Z[T].
\end{equation}

Suppose moreover that $(x,y)\in\cS(P,Q)$ and that, for some $K\in\Z$,
\begin{equation}\label{eq:oriented-companion-surface}
  P(x)+\kappa Q(y)-a=Kxy.
\end{equation}
If
\[
  y'\defeq\frac{P(x)}y\ne0,
\]
then $(x,y')\in\cS(P,C)$ and there exists $K'\in\Z$ such that
\begin{equation}\label{eq:propagated-companion-surface}
  \mu P(x)+C(y')-c=K'xy'.
\end{equation}
Equivalently, after swapping coordinates and coefficient polynomials,
$(y',x)\in\cS(C,P)$ lies on a companion surface of the same form, with
new multiplier $\mu=C(0)/P(0)$.
\end{proposition}

\begin{proof}
Write
\[
  Q(T)=b+\sum_{i=1}^{m}q_iT^i.
\]
Since $a=\kappa b$,
\[
  C(T)
  =T^m+\sum_{i=1}^{m}q_i\kappa^i b^{i-1}T^{m-i},
\]
so $C$ is monic integral, hence primitive.  Its constant term is
$c=q_m\kappa^m b^{m-1}$, which gives
\eqref{eq:companion-multiplier-update}.

To prove \eqref{eq:companion-polynomial-identity}, work first in
$\mathbb Q[T]/(Q)$.  Since $b\ne0$, the class of $T$ is invertible, and
\[
  TH_Q(T)=Q(T)-b\equiv-b\pmod Q.
\]
Thus
\[
  -\kappa H_Q(T)\equiv\frac aT\pmod Q.
\]
By \eqref{eq:companion-C-def},
\[
  C\!\left(-\kappa H_Q(T)\right)
  \equiv C(a/T)
  =\frac1b\left(\frac aT\right)^mQ(T)
  \equiv0\pmod Q.
\]
As $C(U)=c+UH_C(U)$, it follows that
\[
  H_C\!\left(-\kappa H_Q(T)\right)
  \equiv-\frac caT=-\mu T\pmod Q.
\]
The dividend in \eqref{eq:companion-polynomial-identity} lies in
$\Z[T]$; since $Q$ is primitive, Gauss's lemma promotes divisibility in
$\mathbb Q[T]$ to divisibility in $\Z[T]$.

Now divide \eqref{eq:oriented-companion-surface} by $y$.  Using
$Q(y)=b+yH_Q(y)$ and $a=\kappa b$ gives
\begin{equation}\label{eq:companion-coordinate-update}
  y'=Kx-\kappa H_Q(y).
\end{equation}
Hence
$y'\equiv-\kappa H_Q(y)\pmod x$.  Since $x\mid Q(y)$,
\eqref{eq:companion-polynomial-identity} implies
\[
  x\mid \mu y+H_C(y').
\]
Define
\[
  K'\defeq\frac{\mu y+H_C(y')}{x}\in\Z.
\]
Multiplication by $y'$ and the identity $yy'=P(x)$ yield
\[
  K'xy'
  =\mu P(x)+y'H_C(y')
  =\mu P(x)+C(y')-c,
\]
which is \eqref{eq:propagated-companion-surface}.  Finally,
$P(x)\equiv a\pmod x$ and $c=\mu a$, so reducing
\eqref{eq:propagated-companion-surface} modulo $x$ gives
$x\mid C(y')$.  Together with $y'\mid P(x)$ this proves
$(x,y')\in\cS(P,C)$.
\end{proof}

\begin{corollary}[indefinite companion-surface lifting]\label{cor:companion-indefinite}
Let $P,Q\in\Z[T]$ be primitive of degrees at least two, with nonzero
constant terms and
\[
  \kappa\defeq\frac{P(0)}{Q(0)}\in\Z\setminus\{0\}.
\]
Suppose $(x,y)\in\cS(P,Q)$ satisfies
\[
  P(x)+\kappa Q(y)-P(0)=Kxy
\]
for some $K\in\Z$.  Then the alternating canonical monic reciprocal
jumps can be continued through primitive integer coefficient states for
as long as the successive quotient coordinates are nonzero.  At each
step the current constant-term ratio remains a nonzero integer; if the
polynomial being reciprocally transformed is $Q_j$ of degree $m_j$, with
current ratio $\kappa_j$, then
\begin{equation}\label{eq:companion-multiplier-recurrence}
  \kappa_{j+1}
  =\operatorname{lc}(Q_j)\,
   \kappa_j^{m_j-1}Q_j(0)^{m_j-2}.
\end{equation}

If, in addition, neither $P$ nor $Q$ has a rational root, then no quotient
coordinate can vanish, so the integral jump chain is infinite.
\end{corollary}

\begin{proof}
Proposition~\ref{prop:companion-surface-propagation}, followed by swapping
the two coordinates and coefficient polynomials, reproduces its own
hypotheses.  Formula \eqref{eq:companion-multiplier-update} gives
\eqref{eq:companion-multiplier-recurrence}; because every degree is at
least two, its right-hand side is always an integer.  Induction therefore
gives the asserted continuation through monic, hence primitive, integer
coefficient states.

It remains only to exclude zero quotients under the final hypothesis.
For a normalized reciprocal transform
\[
  C(T)=\frac1b\Rcal_aQ(T),
\]
one has, for $s\in\mathbb Q$, $C(s)=0$ if and only if
$Q(a/s)=0$; here $s\ne0$ because $C(0)\ne0$.  Thus normalized reciprocal
transformation preserves the existence or nonexistence of rational
roots.  Every coefficient polynomial occurring later in the alternating
chain is obtained from one of the initial two by such transformations.
Since every current coordinate is an integer, no current polynomial can
vanish there, and all quotient coordinates remain nonzero.
\end{proof}

\begin{corollary}[common-constant indefinite lifting]\label{cor:common-constant-indefinite}
Let $A,B\in\Z[T]$ be primitive of degrees at least two and satisfy
$A(0)=B(0)=c\ne0$.  Every integral point with $xy\ne0$ on
\[
  A(x)+B(y)-c=zxy
\]
starts a canonical monic congruence-jump chain that remains in primitive
integer coefficient states for as long as the quotient coordinates are
nonzero.  If neither $A$ nor $B$ has a rational root, this chain is
infinite.
\end{corollary}

\begin{proof}
Corollary~\ref{cor:surface-normalization} puts the initial point in
$\cS(A,B)$, and the companion multiplier is $\kappa=1$.  Apply
Corollary~\ref{cor:companion-indefinite}.
\end{proof}

For coefficient-level statements, call two nonzero polynomials $P,Q\in\mathbb Q[T]$
\emph{rationally gauge equivalent} when
\[
  Q(T)=\lambda P(\mu T)
\]
for some $\lambda,\mu\in\mathbb Q^{\times}$.

\subsection{A four-step rational gauge identity}\label{subsec:four-step-block}

\begin{proposition}[four-step rational gauge closure]\label{prop:four-step-gauge-closure}
Let $A,B\in\Z[T]$ have positive degrees $n,m$, respectively, and nonzero
constant terms
\[
  a\defeq A(0),\qquad b\defeq B(0).
\]
For coefficient dynamics, extend the definition of $\Rcal_cP$ from
\eqref{eq:c-reciprocal} to $c\in\mathbb Q^{\times}$ and
$P\in\mathbb Q[T]$ by the same formula.  Define successively
\begin{align}
  B_1&\defeq\frac1b\Rcal_aB,
  &q&\defeq B_1(0),\label{eq:gauge-four-B1}\\
  A_1&\defeq\frac1a\Rcal_qA,
  &r&\defeq A_1(0),\label{eq:gauge-four-A1}\\
  B_2&\defeq\frac1q\Rcal_rB_1,
  &t&\defeq B_2(0),\label{eq:gauge-four-B2}\\
  A_2&\defeq\frac1r\Rcal_tA_1,
  &u&\defeq A_2(0).\label{eq:gauge-four-A2}
\end{align}
Then $B_1,A_1,B_2,A_2\in\mathbb Q[T]$ are monic.  Writing
$a_n\defeq\operatorname{lc}(A)$ and $b_m\defeq\operatorname{lc}(B)$,
their successive constant terms are
\begin{equation}\label{eq:gauge-four-constants}
  q=\frac{b_ma^m}{b},\qquad
  r=\frac{a_nq^n}{a},\qquad
  t=\frac{r^m}{q},\qquad
  u=\frac{t^n}{r}.
\end{equation}
Moreover,
\begin{equation}\label{eq:gauge-four-return}
  B_2(T)=\frac{t}{b}\,B\!\left(\frac{a}{r}T\right),
  \qquad
  A_2(T)=\frac{u}{a}\,A\!\left(\frac{q}{t}T\right).
\end{equation}
Thus every coefficient pair with nonzero constant terms returns after four
monic rational reciprocal normalizations to its original pair of rational
gauge classes.  This coefficient-level identity will be used only as a bookkeeping
device for the arithmetic lifts below; it does not by itself imply pointwise return.
\end{proposition}

\begin{proof}
The leading coefficient of $\Rcal_cP$ is $P(0)$.  Hence each division in
\eqref{eq:gauge-four-B1}--\eqref{eq:gauge-four-A2} makes the resulting
polynomial monic, and taking constant terms gives
\eqref{eq:gauge-four-constants}.

For any degree-$d$ polynomial $P\in\mathbb Q[T]$ and nonzero
$v,w\in\mathbb Q$,
\begin{equation}\label{eq:double-reciprocal-gauge}
  \Rcal_w(\Rcal_vP)(T)=w^dP\!\left(\frac vwT\right).
\end{equation}
Using \eqref{eq:double-reciprocal-gauge} twice gives
\[
  B_2(T)
  =\frac{1}{qb}\Rcal_r(\Rcal_aB)(T)
  =\frac{r^m}{qb}B\!\left(\frac arT\right)
  =\frac tbB\!\left(\frac arT\right),
\]
and similarly
\[
  A_2(T)
  =\frac{1}{ra}\Rcal_t(\Rcal_qA)(T)
  =\frac{t^n}{ra}A\!\left(\frac qtT\right)
  =\frac uaA\!\left(\frac qtT\right).
\]
This proves \eqref{eq:gauge-four-return}.
\end{proof}

\subsection{Surface-assisted arithmetic lifting}\label{subsec:surface-assisted}

\begin{example}[surface-assisted cancellation at $c=5$]\label{ex:surface-cancellation}
Let
\[
  A(T)\defeq2T^5+T^2+5,
  \qquad
  B(T)\defeq T^3-3T+5.
\]
The common-constant surface \eqref{eq:common-constant-surface} is
\begin{equation}\label{eq:c5-surface}
  2x^5+x^2+y^3-3y+5=zxy.
\end{equation}
Here
\[
  \Rcal_5B(T)
  =5T^3-15T^2+125
  =5\bigl(T^3-3T^2+25\bigr),
\]
so
\[
  C_B(T)=T^3-3T^2+25.
\]
Consequently every nonzero integer point of \eqref{eq:c5-surface} with
$y'=A(x)/y\ne0$ satisfies
\begin{equation}\label{eq:c5-normalized-congruence}
  x\mid (y')^3-3(y')^2+25.
\end{equation}
Thus the factor $5$ in the raw reciprocal transform can be cancelled without assuming $\gcd(x,5)=1$.  Directly, \eqref{eq:c5-surface} becomes
\[
  y'+y^2-3=zx;
\]
multiplying by $(y')^2$ and using $yy'=A(x)\equiv5\pmod x$ gives \eqref{eq:c5-normalized-congruence}.
\end{example}

\begin{remark}[multiplicative surface modifications]\label{rem:multiplicative-modifications}
Let $(x,y,z)$ satisfy \eqref{eq:common-constant-surface}, and let $V\in\Z[T]$ with $V(0)=1$.  Since $V(y)-1$ is divisible by $y$ and $x\mid B(y)$, there is an integer $z_V$ such that the same pair $(x,y)$ lies on the auxiliary surface
\[
  A(x)+V(y)B(y)-c=z_Vxy.
\]
Equivalently, this adds to the original surface equation a term $f(y)B(y)$ with $f(0)=0$.  Applying Corollary~\ref{cor:surface-normalization} to the auxiliary surface gives a transformed congruence whose polynomial factors canonically, because the reciprocal transform is multiplicative:
\begin{equation}\label{eq:multiplicative-reciprocal-modification}
  \frac1c\Rcal_c(VB)
  =\Rcal_cV\,C_B.
\end{equation}
Hence this auxiliary congruence is already implied by $x\mid C_B(y')$ and supplies no independent restriction at the first $y$-jump.  The analogous statement holds for a modification in the $A$-coordinate.

For Example~\ref{ex:surface-cancellation}, taking $V(T)=1+T$ gives
\[
  \Rcal_5V(T)=T+5
\]
and
\begin{align*}
  \frac15\Rcal_5(VB)(T)
  &=(T+5)(T^3-3T^2+25)\\
  &=T^4+2T^3-15T^2+25T+125.
\end{align*}
Thus a Euclidean comparison of this quartic with $C_B$ has zero remainder rather than producing an additional restriction.
\end{remark}

\begin{proposition}[infinite arithmetic lift at $c=5$]\label{prop:c5-infinite-lift}\label{prop:c5-four-step-lift}
Continue Example~\ref{ex:surface-cancellation}, and write
\[
  A_0(T)\defeq2T^5+T^2+5,
  \qquad
  B_0(T)\defeq T^3-3T+5.
\]
Every integral point $(x,y,z)$ of \eqref{eq:c5-surface} with $xy\ne0$
generates an infinite sequence of nonzero integer quotient coordinates
under alternating canonical monic reciprocal jumps.

The first four normalized coefficient states are obtained from
\[
  B_1(T)\defeq T^3-3T^2+25,
  \qquad
  A_1(T)\defeq T^5+125T^3+2\cdot5^9,
\]
\[
  B_2(T)\defeq T^3-12\cdot5^{16}T+8\cdot5^{25},
\]
\[
  A_2(T)\defeq
  T^5+2^8 5^{69}T^2+2^{14}5^{116}.
\]
In particular, if
\begin{equation}\label{eq:c5-four-coordinates}
  y_1\defeq\frac{A_0(x)}y,
  \qquad
  x_1\defeq\frac{B_1(y_1)}x,
  \qquad
  y_2\defeq\frac{A_1(x_1)}{y_1},
  \qquad
  x_2\defeq\frac{B_2(y_2)}{x_1},
\end{equation}
then all four coordinates are nonzero integers and
\begin{equation}\label{eq:c5-four-states}
\begin{aligned}
 (x,y)&\in\cS(A_0,B_0),\\
 (x,y_1)&\in\cS(A_0,B_1),\\
 (x_1,y_1)&\in\cS(A_1,B_1),\\
 (x_1,y_2)&\in\cS(A_1,B_2),\\
 (x_2,y_2)&\in\cS(A_2,B_2).
\end{aligned}
\end{equation}
The final coefficient pair of this first four-step block is the explicit
rational gauge return
\begin{equation}\label{eq:c5-gauge-return}
  B_2(T)=8\cdot5^{24}
  B_0\!\left(\frac{T}{2\cdot5^8}\right),
  \qquad
  A_2(T)=2^{14}5^{115}
  A_0\!\left(\frac{T}{8\cdot5^{23}}\right).
\end{equation}
\end{proposition}

\begin{proof}
The initial point belongs to $\cS(A_0,B_0)$ by
Corollary~\ref{cor:surface-normalization}, and its companion multiplier is
$\kappa_0=1$.  Both polynomials are primitive and have degrees $5$ and
$3$, so Corollary~\ref{cor:common-constant-indefinite} gives an integral
primitive jump chain for as long as no quotient coordinate vanishes.

Neither $A_0$ nor $B_0$ has a rational root.  Indeed, the rational-root
theorem leaves only
\[
  \pm1,\ \pm5,\ \pm\frac12,\ \pm\frac52
\]
for $A_0$ and only $\pm1,\pm5$ for $B_0$, and direct substitution excludes
all of them.  Corollary~\ref{cor:companion-indefinite} therefore shows
that no later reciprocal state has a rational root.  Since every current
coordinate is integral, no numerator in the quotient chain can vanish.
Thus the chain is infinite and all its coordinates are nonzero integers.

The displayed formulas for $B_1,A_1,B_2,A_2$ follow from the canonical
monic reciprocal normalization.  For the first four-step block the
successive constants in Proposition~\ref{prop:four-step-gauge-closure}
are
\[
  q=25,\qquad r=2\cdot5^9,\qquad t=8\cdot5^{25},
\]
and substitution in \eqref{eq:gauge-four-return} gives
\eqref{eq:c5-gauge-return}.
\end{proof}

\begin{remark}[growth of the companion multipliers]\label{rem:c5-companion-multipliers}
For the $c=5$ chain, the multiplier recurrence
\eqref{eq:companion-multiplier-recurrence} begins
\begin{equation}\label{eq:c5-companion-multipliers}
  1,\qquad
  5,\qquad
  2\cdot5^7,\qquad
  2^2 5^{16},\qquad
  2^{11}5^{91},\ldots.
\end{equation}
Thus the auxiliary companion surfaces change rapidly even though the
canonical divisibility chain continues integrally forever.  This is the
arithmetic mechanism behind the ``jump through changing surfaces''
phenomenon in this example.
\end{remark}

\begin{example}[a finite chain stopped by normalization]\label{ex:normalization-stops-chain}
Let
\[
  Q(T)\defeq T^2-2T+2,
  \qquad
  \Rcal_2Q=2Q.
\]
For positive integers $a_j$, put $b_j\defeq a_j+1$.  Then
\[
  a_{j+1}+1=\frac{a_j^2+1}{a_{j-1}+1}
  \qquad\Longleftrightarrow\qquad
  b_{j-1}b_{j+1}=Q(b_j).
\]
Thus a four-term chain makes $(b_2,b_3)\in\cS(Q,Q)$, but after the right jump Theorem~\ref{thm:local-reciprocal} guarantees only
$b_3\mid2Q(b_4)$; a fifth term requires the primitive normalization
$b_3\mid Q(b_4)$.

For $(b_1,b_2,b_3,b_4)=(5,34,218,1385)$,
\[
  Q(34)=5\cdot218,
  \qquad Q(218)=34\cdot1385,
  \qquad Q(1385)\equiv109\pmod{218}.
\]
Hence $218\nmid Q(1385)$ although $218\mid2Q(1385)$: the normalization obstruction is exactly the prime $2$.  The chain
$(13,58,250,1069)$ behaves similarly, with
$Q(1069)\equiv125\pmod{250}$.  Proposition~\ref{prop:quadratic-denominator-rigidity} will show that both four-term chains are maximal.
\end{example}

\section{Finite reciprocal cycles}\label{sec:four-jump}

\subsection{Unit and signed endpoint cycles}\label{subsec:unit-signed-cycles}

\begin{proposition}[unit self-reciprocal specialization]\label{prop:unit-selfreciprocal}
Let $A,B\in\Z[T]$ satisfy
\[
  A(0)=B(0)=1,
  \qquad A^{\vee}=A,
  \qquad B^{\vee}=B.
\]
Then every nonzero quotient jump of Theorem~\ref{thm:local-reciprocal} remains in the same state $\cS(A,B)$ and is reversible.
\end{proposition}

\begin{proof}
A jump replaces the opposite polynomial by its ordinary reciprocal, hence leaves the state unchanged.  If $y'=A(x)/y\ne0$, then $A(x)/y'=y$, and similarly in the other coordinate.
\end{proof}

\begin{remark}[classical fixed-state precedents]\label{rem:classical-selfreciprocal}
Mohanty's system $x\mid y^3+1$, $y\mid x^3+1$ is the cubic specialization of Proposition~\ref{prop:unit-selfreciprocal}; its chains satisfy $u_{j-1}u_{j+1}=u_j^3+1$~\cite{Mohanty1977}.  Bier's family $T^h+1,T^k+1$ gives the same alternating self-reciprocal mechanism in general degrees and is part of a broader finite-versus-infinite theory for $x\mid g(y)$, $y\mid f(x)$~\cite{Bier2014}.  Dofs's mixed system
\[
  x\mid y^2-y+1,\qquad y\mid x^3+1
\]
shows that reversibility does not imply connectedness: he exhibits positive solutions, including $\{49,325\}$ and $\{49,362\}$, outside the orbit of $(1,1)$~\cite{Dofs1993}.
\end{remark}

We first specialize to the normalized unit case
\begin{equation}\label{eq:unit-endpoints}
  A(0)=B(0)=1,
  \qquad
  \operatorname{lc}(A)=\operatorname{lc}(B)=1.
\end{equation}
Then $\Rcal_1A=A^{\vee}$ and $\Rcal_1B=B^{\vee}$, while reciprocity preserves \eqref{eq:unit-endpoints} and is involutive:
\[
  (A^{\vee})^{\vee}=A,
  \qquad
  (B^{\vee})^{\vee}=B.
\]

The constant terms in \eqref{eq:unit-endpoints} force coprimality.

\begin{lemma}\label{lem:coprime}
Every $(x,y)\in\cS(A,B)$ satisfies $\gcd(x,y)=1$.
\end{lemma}

\begin{proof}
Let $d$ divide both $x$ and $y$.  Since $y\mid A(x)$, we have $d\mid A(x)$.  On the other hand,
\[
  A(x)\equiv A(0)\equiv1\pmod d,
\]
so $|d|=1$.
\end{proof}

Consequently each normalized divisibility state can be viewed as a union of integral points on a family of affine surfaces.

\begin{proposition}\label{prop:k-surface}
For every $(x,y)\in\cS(A,B)$, the number
\begin{equation}\label{eq:k-def}
  k(x,y)\defeq\frac{A(x)+B(y)-1}{xy}
\end{equation}
is an integer.  Thus $(x,y)$ determines an integral point on
\begin{equation}\label{eq:surface-k}
  A(x)+B(y)-1=kxy.
\end{equation}
\end{proposition}

\begin{proof}
Modulo $y$, the numerator in \eqref{eq:k-def} is congruent to $A(x)$, and modulo $x$ it is congruent to $B(y)$.  Hence it is divisible by both $x$ and $y$.  Lemma~\ref{lem:coprime} gives $\gcd(x,y)=1$, so it is divisible by $xy$.
\end{proof}

The point of congruence jumping is that the integer $k(x,y)$ need not be invariant.

Under \eqref{eq:unit-endpoints}, Theorem~\ref{thm:local-reciprocal} becomes ordinary reciprocal jumping.  If
\[
  y'\defeq\frac{A(x)}{y},
\]
then $(x,y')\in\cS(A,B^{\vee})$; symmetrically, $x':=B(y)/x$ gives $(x',y)\in\cS(A^{\vee},B)$.  Because the reciprocal polynomials again have constant and leading coefficient $1$, the operation can be iterated without any renormalization.

\begin{corollary}[four-jump reciprocal cycle]\label{thm:four-jump}
Let $A,B$ satisfy \eqref{eq:unit-endpoints}, and let $(x_0,y_0)\in\cS(A,B)$ be nondegenerate for the following four quotients:
\begin{align}
  y_1 &\defeq \frac{A(x_0)}{y_0},\label{eq:cycle1}\\
  x_1 &\defeq \frac{B^{\vee}(y_1)}{x_0},\label{eq:cycle2}\\
  y_2 &\defeq \frac{A^{\vee}(x_1)}{y_1},\label{eq:cycle3}\\
  x_2 &\defeq \frac{B(y_2)}{x_1}.\label{eq:cycle4}
\end{align}
Then all four quotients are integers and the successive states are
\begin{equation}\label{eq:state-cycle}
\begin{aligned}
 (x_0,y_0)&\in\cS(A,B),\\
 (x_0,y_1)&\in\cS(A,B^{\vee}),\\
 (x_1,y_1)&\in\cS(A^{\vee},B^{\vee}),\\
 (x_1,y_2)&\in\cS(A^{\vee},B),\\
 (x_2,y_2)&\in\cS(A,B).
\end{aligned}
\end{equation}
Hence the composition \eqref{eq:cycle1}--\eqref{eq:cycle4} returns to the original divisibility system.
\end{corollary}

\begin{proof}
This is Theorem~\ref{thm:local-reciprocal} applied four times, using $(A^{\vee})^{\vee}=A$ and $(B^{\vee})^{\vee}=B$.
\end{proof}

\begin{proposition}[signed endpoint state space]\label{prop:signed-states}
Suppose that
\[
  A(0),\operatorname{lc}(A),B(0),\operatorname{lc}(B)\in\{\pm1\}.
\]
Under any sequence of alternating quotient jumps supplied by Theorem~\ref{thm:local-reciprocal}, every polynomial that replaces $A$ is, up to multiplication by $-1$, one of
\[
  A(T),\qquad A(-T),\qquad A^{\vee}(T),\qquad A^{\vee}(-T),
\]
and the analogous statement holds for $B$.  Consequently the coefficient pair ranges through at most $16$ divisibility states modulo independent multiplication of the two defining polynomials by $-1$.  If the alternating jumps can be continued indefinitely, their coefficient-state sequence is periodic.
\end{proposition}

\begin{proof}
For a degree-$d$ polynomial $P$,
\[
  \Rcal_1P=P^{\vee},
  \qquad
  \Rcal_{-1}P(T)=(-1)^dP^{\vee}(-T).
\]
Hence, modulo an overall sign, every transform of $P$ is one of the four displayed forms.  Keeping the overall sign gives at most eight actual polynomial representatives for each initial polynomial, hence at most $64$ signed pairs.  Each step is reversible on this finite signed state space because
\[
  \Rcal_{\epsilon}^2P=\epsilon^dP
  \qquad(\epsilon=\pm1),
\]
and the polynomial not being transformed records the sign $\epsilon$ used at that step.  Therefore the two-step map (which fixes the jump parity) is a permutation of a finite set.  Its orbit is periodic from the initial state.  Passing to divisibility states modulo independent overall signs leaves at most $16$ states and preserves periodicity.
\end{proof}

\begin{remark}\label{rem:signed-endpoints}
The normalized assumption \eqref{eq:unit-endpoints} is therefore not needed for finiteness of the coefficient-state dynamics.  Its role in Corollary~\ref{thm:four-jump} is to make the return a literal four-step square under ordinary reciprocity, rather than a possibly longer cycle through signed reciprocal forms.
\end{remark}

\begin{remark}[equation versus congruence]\label{rem:k-changes}
Corollary~\ref{thm:four-jump} returns from $\cS(A,B)$ to $\cS(A,B)$, but it does not assert that
\[
  k(x_2,y_2)=k(x_0,y_0)
\]
for the integer in \eqref{eq:k-def}.  Thus the cycle is naturally a reversible self-map of the nondegenerate locus of the divisibility problem, or of the corresponding nondegenerate locus in the union of surfaces \eqref{eq:surface-k} over $k$, rather than necessarily a self-map of each fixed-$k$ surface.
\end{remark}

\subsection{A three-jump power family}\label{subsec:three-jump-power}

\begin{proposition}[three-jump return for $T-1$ and $T^d-1$]\label{prop:three-jump-power}
Let $d\ge2$, put
\[
  A(T)\defeq T-1,
  \qquad
  B(T)\defeq T^d-1,
\]
and let $(p,q)\in\cS(A,B)$ with $p,q>1$.  Thus
\[q\mid p-1,
  \qquad
  p\mid q^d-1.
\]
Define
\begin{equation}\label{eq:three-jump-ukp}
  u\defeq\frac{q^d-1}{p},
  \qquad
  k\defeq\frac{u+1}{q},
  \qquad
  p'\defeq\frac{k^d-1}{u}.
\end{equation}
Then $u,k\in\Z_{>0}$ and $p'\in\Z_{\ge0}$.  If $p'\ne0$, then
\[
  (p',k)\in\cS(A,B),
\]
and the passage $(p,q)\mapsto(p',k)$ is the composition of three quotient jumps through the coefficient states
\begin{equation}\label{eq:three-jump-states}
  \cS(T-1,T^d-1)
  \longrightarrow
  \cS(T+1,T^d-1)
  \longrightarrow
  \cS(T+1,T^d-1)
  \longrightarrow
  \cS(T-1,T^d-1).
\end{equation}
On the nondegenerate locus this three-jump return is an involution.
\end{proposition}

\begin{proof}
The first quotient $u$ is integral by hypothesis.  Since
\[
  pu=q^d-1,
  \qquad p\equiv1\pmod q,
\]
one has $u\equiv-1\pmod q$, so $k$ is integral.  Next
$kq=u+1\equiv1\pmod u$, while $q^d\equiv1\pmod u$; hence
$k^d\equiv1\pmod u$ and $p'$ is integral.  Finally,
$p'u=k^d-1$ and $u\equiv-1\pmod k$, so $p'\equiv1\pmod k$.
Thus $k\mid p'-1$ and $p'\mid k^d-1$ whenever $p'\ne0$.

The state cycle follows from
\[
  \Rcal_{-1}(T-1)=-(T+1),\qquad
  \Rcal_1(T^d-1)=1-T^d,\qquad
  \Rcal_{-1}(T+1)=T-1,
\]
where an overall sign does not affect a divisibility state.  Starting from $(p',k)$, the same three quotients are $u,q,p$ in reverse order, proving involutivity.
\end{proof}

\paragraph{Relation with companion surfaces.}

The reciprocal cycle is the congruence-level counterpart of the companion-surface groupoid of Koll\'ar--Villalobos-Paz~\cite{KollarVillalobosPaz2025}, used by Koll\'ar--Li in the cubic Mordell--Schinzel problem~\cite{KollarLi2024}.  For
\[
  P(T)=p_dT^d+\cdots+p_0,
\]
their normalized reciprocal is
\begin{equation}\label{eq:kvp-reciprocal}
  \overline P(T)
  \defeq
  \frac{1}{p_0^{d-1}p_d}\,\Rcal_{p_0}P(T).
\end{equation}
When the endpoint coefficients are units, the normalization supports global isomorphisms among four companion surfaces.  Their four-dimensional model
\[
  yu=A(x),\qquad xv=B(y)
\]
already contains the same quotient coordinates.

The difference is local versus global.  If $p_0$ or $p_d$ is a nonunit, the normalization in \eqref{eq:kvp-reciprocal} need not be regular over the base ring, whereas Theorem~\ref{thm:local-reciprocal} retains the integral polynomial $\Rcal_{p_0}P$ and still gives a pointwise congruence implication.  The nonunit problem is therefore not existence of a first jump but closure and normalization of the subsequent coefficient states.

\section{Quadratic chains: denominator rigidity and Vieta reduction}\label{sec:quadratic}

\subsection{Finite-chain rigidity and the Vieta bridge}\label{subsec:quadratic-rigidity}

Degree two contains two complementary layers.  First, a long integral quotient chain constrains the denominator of its rational fixed-conic parameter; we begin with this finite arithmetic rigidity.  Second, once the parameter is integral, the usual Vieta companion-root dynamics and its linearization apply.  The fixed-conic mechanism itself is classical: Mills's 1954 paper treats $x^2+\varepsilon y^2+ax+by+c=xyz$ by alternating companion roots in $x$ and $y$~\cite{Mills1954}, and his 1956 generalization gives coefficient divisibility hypotheses ensuring integrality for a general quadratic numerator and bilinear-affine denominator~\cite{Mills1956}.

The associated linearization is also classical.  Alperin considers the recurrence
\[
  x_{n-1}x_{n+1}=x_n^2+r x_n+c
\]
and observes that it is linearly recursive of degree three~\cite[Sec.~2.2]{Alperin2011}, with characteristic polynomial
\[
  X^3-\mu X^2+\mu X-1,
  \qquad
  \mu=
  \frac{x_{n-1}^2+x_n^2+x_{n-1}x_n+r(x_{n-1}+x_n)+c}
       {x_{n-1}x_n}.
\]
In our notation below, $\mu=H+1$, so this polynomial factors as
\[
  (X-1)(X^2-HX+1).
\]
Thus the affine Vieta recurrence and the homogeneous second-order recurrence for first differences are not new.  The quantitative point needed here is instead finite: if $H=A/B$ is rational and a segment remains integral for $L$ terms, powers of $B$ are forced into a fixed Cassini-type determinant.  This gives a finite denominator-length obstruction before any infinite-sequence argument is available.

\begin{proposition}[denominator rigidity for finite quadratic chains]\label{prop:quadratic-denominator-rigidity}
Let
\[
  Q(T)\defeq T^2+rT+c,
  \qquad r,c\in\Z,
\]
and let $x_0,x_1,\ldots,x_N$ be nonzero integers, with $N\ge3$, satisfying
\begin{equation}\label{eq:finite-quadratic-chain}
  x_{n-1}x_{n+1}=Q(x_n)
  \qquad(1\le n\le N-1).
\end{equation}
Then the rational number
\begin{equation}\label{eq:finite-chain-H}
  H\defeq
  \frac{x_{n-1}^2+x_n^2+r(x_{n-1}+x_n)+c}{x_{n-1}x_n}
\end{equation}
is independent of $n$.  Write $H=A/B$ in lowest terms, with $B>0$, and put
\[
  d_n\defeq x_n-x_{n-1}
  \qquad(1\le n\le N).
\]
Then
\begin{equation}\label{eq:difference-denominator-divisibility}
  B^{\min(n-1,N-n)}\mid d_n
  \qquad(1\le n\le N).
\end{equation}
Moreover,
\begin{equation}\label{eq:Theta-definition}
  \Theta\defeq d_n d_{n+2}-d_{n+1}^2
\end{equation}
is independent of $n$ and satisfies
\begin{equation}\label{eq:Theta-conic-identity}
  \Theta=(2-H)c-r^2\in\Z,
\end{equation}
together with the quantitative divisibility
\begin{equation}\label{eq:Theta-denominator-bound}
  B^{N-3}\mid\Theta.
\end{equation}
In particular, $B\mid c$.  If $B>1$ and $\Theta\ne0$, a chain of
$L=N+1$ terms satisfies
\begin{equation}\label{eq:finite-chain-length-bound}
  L\le
  4+
  \min_{p\mid B}
  \left\lfloor\frac{v_p(\Theta)}{v_p(B)}\right\rfloor.
\end{equation}
\end{proposition}

\begin{proof}
From \eqref{eq:finite-quadratic-chain}, the value of the fraction in
\eqref{eq:finite-chain-H} at the adjacent pair $(x_{n-1},x_n)$ is
\[
  \frac{x_{n-1}+x_{n+1}+r}{x_n},
\]
and the same simplification at $(x_n,x_{n+1})$ gives the identical value.  Hence
\eqref{eq:finite-chain-H} is constant and
\begin{equation}\label{eq:finite-affine-recurrence}
  x_{n+1}=Hx_n-x_{n-1}-r.
\end{equation}
Hence the differences satisfy
\begin{equation}\label{eq:difference-rational-recurrence}
  B d_{n+1}=A d_n-B d_{n-1}
  \qquad(2\le n\le N-1).
\end{equation}
Reducing modulo $B$ first gives $B\mid d_2,\ldots,d_{N-1}$.  More generally, if
$B^j\mid d_{n-1}$ and $B^j\mid d_{n+1}$, then
\eqref{eq:difference-rational-recurrence} gives
\[
  A d_n=B(d_{n-1}+d_{n+1}),
\]
so $B^{j+1}\mid d_n$ because $\gcd(A,B)=1$.  Starting with $j=0$ at the two boundary differences and iterating this inward proves
\eqref{eq:difference-denominator-divisibility}.

The homogeneous recurrence for the differences is
$d_{n+1}=Hd_n-d_{n-1}$.  It follows immediately that
$d_n d_{n+2}-d_{n+1}^2$ is independent of $n$.  Substituting
\eqref{eq:finite-affine-recurrence} into this determinant and using the conic equation gives
\eqref{eq:Theta-conic-identity}.  Since the left side is integral, and
$\gcd(A-2B,B)=1$, the identity also gives $B\mid c$.

Finally, apply \eqref{eq:Theta-definition} at a central index.  If $N=2m$, take
$n=m-1$; then the two terms $d_n d_{n+2}$ and $d_{n+1}^2$ are divisible by
$B^{2m-3}=B^{N-3}$.  If $N=2m+1$, take $n=m$; both terms are divisible by
$B^{2m-2}=B^{N-3}$.  This proves \eqref{eq:Theta-denominator-bound}.  The length bound
\eqref{eq:finite-chain-length-bound} is the corresponding $p$-adic restatement when
$\Theta\ne0$.
\end{proof}

\begin{corollary}[unit constant term]\label{cor:quadratic-unit-constant}
If $c=\pm1$, every four-term nonzero integral chain for
$Q(T)=T^2+rT+c$ has $H\in\Z$.
\end{corollary}

\begin{proof}
For four terms one has $N=3$, and Proposition~\ref{prop:quadratic-denominator-rigidity}
gives $B\mid c$.  Since $|c|=1$, this forces $B=1$.
\end{proof}

\begin{remark}[the stopped chains revisited]\label{rem:stopped-chains-rigidity}
For Example~\ref{ex:normalization-stops-chain}, $r=-2$ and $c=2$.  The two four-term chains have $H=13/2$ and $H=9/2$, hence
\[
  \Theta=(2-H)c-r^2=-13,\qquad -9,
\]
respectively.  A five-term continuation would require $2\mid\Theta$ by Proposition~\ref{prop:quadratic-denominator-rigidity}, impossible in either case.  Thus the local failure to cancel the scalar $2$ is the first manifestation of the global denominator-length obstruction.
\end{remark}

\begin{corollary}[one-sided and bi-infinite rigidity]\label{cor:quadratic-infinite-rigidity}
Let $Q(T)=T^2+rT+c\in\Z[T]$ and let $(x_n)$ be a nonconstant sequence of nonzero integers satisfying
\begin{equation}\label{eq:infinite-quadratic-chain}
  x_{n-1}x_{n+1}=Q(x_n).
\end{equation}

\begin{enumerate}[label=\textup{(\roman*)}]
\item Suppose $n\ge0$ and the chain is one-sided infinite.  If its invariant
$H=A/B$ is not an integer, then there are an integer $m$ with $|m|\ge2$ and
$\alpha\in\Z$ such that
\begin{equation}\label{eq:exceptional-factorization}
  H=m+\frac1m,
  \qquad
  Q(T)=(T-\alpha)(T+m\alpha),
\end{equation}
and
\begin{equation}\label{eq:exceptional-geometric-chain}
  x_n
  =m^n x_0
   -m\alpha\frac{m^n-1}{m-1}
  \qquad(n\ge0).
\end{equation}
Conversely, whenever the terms in \eqref{eq:exceptional-geometric-chain} are nonzero, they satisfy \eqref{eq:infinite-quadratic-chain} for the polynomial in \eqref{eq:exceptional-factorization}.

\item Consequently, if $Q$ is irreducible over $\Z$, every nonconstant one-sided infinite integral chain has $H\in\Z$.

\item Every nonconstant bi-infinite integral chain has $H\in\Z$, without an irreducibility hypothesis.
\end{enumerate}
\end{corollary}

\begin{proof}
Assume first that the chain is one-sided infinite and $B>1$.  Proposition~\ref{prop:quadratic-denominator-rigidity}, applied to arbitrarily long initial segments, forces $\Theta=0$.  If one difference $d_n$ vanished, then
$\Theta=0$ and the difference recurrence would force all differences to vanish, contrary to nonconstancy.  Thus every $d_n$ is nonzero, and
\[
  d_n d_{n+2}=d_{n+1}^2
\]
shows that $d_{n+1}/d_n$ is a constant rational number.  Write it as $m$.  Since
$d_n=d_1m^{n-1}$ is integral for every $n$, the denominator of $m$ would have to divide $d_1$ to arbitrarily high powers; hence $m\in\Z\setminus\{0\}$.  The difference recurrence gives
\[
  H=m+\frac1m.
\]
Since $H\notin\Z$, one has $|m|\ge2$.

Now \eqref{eq:Theta-conic-identity} and $\Theta=0$ give
\[
  (m-1)^2c=-mr^2.
\]
Because $\gcd(m,(m-1)^2)=1$, one has $m-1\mid r$.  Put
$\alpha\defeq r/(m-1)$.  Then $c=-m\alpha^2$, proving
\eqref{eq:exceptional-factorization}.  Comparing $d_2=md_1$ with the affine recurrence yields
\[
  d_1=(m-1)x_0-m\alpha,
\]
and summing the geometric progression of differences gives
\eqref{eq:exceptional-geometric-chain}.  Direct substitution proves the converse.  This proves \textup{(i)}, and \textup{(ii)} follows because the exceptional polynomial factors over $\Z$.

For a bi-infinite chain with nonintegral $H$, part~\textup{(i)} applies to either half.  On the forward half, the nonzero differences satisfy
$d_{n+1}=md_n$ with $|m|\ge2$.  Since $\Theta=0$, the same ratio persists backwards, so $d_{1-k}=d_1/m^k$ is an integer for every $k\ge1$.  This forces $d_1=0$, contradicting nonconstancy.  Hence $H\in\Z$.
\end{proof}

\begin{remark}[Fatou's lemma and infinite rigidity]\label{rem:fatou-rigidity}
There is a classical infinite-series explanation for the genuinely order-two part of Corollary~\ref{cor:quadratic-infinite-rigidity}.  The integer difference sequence satisfies
\[
  d_{n+1}=H d_n-d_{n-1}.
\]
If this is its minimal recurrence, then its reduced generating function has denominator
\[
  1-Ht+t^2.
\]
Fatou's lemma for rational series with integer coefficients implies that this reduced denominator has integer coefficients; hence $H\in\Z$~\cite[Ch.~7, Sec.~1, Cor.~1.2]{BerstelReutenauer2011}.  The only way a nonintegral $H$ can evade this order-two conclusion is that the minimal recurrence drops to order one, so that $d_{n+1}=m d_n$.  Part~\textup{(i)} above identifies exactly which nonlinear quadratic chains realize that degeneration and determines their factorization and explicit orbit.

Proposition~\ref{prop:quadratic-denominator-rigidity} is a finite quantitative precursor to this infinite integrality principle: before a generating-function argument applies, an $L$-term integral segment already forces the denominator power $B^{L-4}$ into the fixed determinant $\Theta$.
\end{remark}

\begin{remark}[sharpness and constant chains]\label{rem:quadratic-rigidity-sharpness}
The one-sided exception is genuine.  For
\[
  Q(T)=T^2-18T-648=(T+18)(T-36),
\]
take $m=2$, $\alpha=-18$, and $x_0=1$.  Formula
\eqref{eq:exceptional-geometric-chain} gives
\[
  x_n=37\cdot2^n-36
  =1,38,112,260,556,1148,\ldots,
\]
a positive one-sided infinite chain with $H=5/2$.

The qualifier ``nonconstant'' is also essential.  A constant chain $x_n=s\ne0$ exists exactly when $rs+c=0$, and then
$H=2+r/s$ may be nonintegral even for an irreducible quadratic.  For example,
$Q(T)=T^2-T+2$ has the constant chain $x_n=2$ with $H=3/2$.

For Mills's family $P_a(T)=T^2+aT+1$, Corollary~\ref{cor:quadratic-unit-constant} shows directly that every four-term integral quotient chain already has integral conic parameter.  This is consistent with the classical affine recurrences in~\cite{Mills1953}.
\end{remark}

\begin{proposition}[quadratic bridge to Vieta jumping]\label{prop:quadratic-vieta-bridge}
Let
\begin{equation}\label{eq:general-monic-Q}
  Q(T)\defeq T^2+rT+c,
  \qquad r,c\in\Z,
\end{equation}
and, for nonzero integers $x,y$, define
\begin{equation}\label{eq:H-general}
  H_Q(x,y)\defeq
  \frac{x^2+y^2+r(x+y)+c}{xy}.
\end{equation}
Then the following hold.

\begin{enumerate}[label=\textup{(\roman*)}]
\item If
\[
  x\mid Q(y),
  \qquad
  y\mid Q(x),
  \qquad
  \gcd(x,y)=1,
\]
then $H_Q(x,y)\in\Z$.

\item Whenever $H\defeq H_Q(x,y)\in\Z$, the quotient
\begin{equation}\label{eq:general-quadratic-jump}
  x'\defeq\frac{Q(y)}x
  =Hy-r-x
\end{equation}
is an integer.  If $x'\ne0$, then
\begin{equation}\label{eq:H-invariance}
  H_Q(y,x')=H,
\end{equation}
and
\[
  y\mid Q(x'),
  \qquad
  x'\mid Q(y).
\]
Thus, on the nondegenerate locus, the quotient jump and the Vieta companion-root jump are the same operation on the fixed conic
\begin{equation}\label{eq:general-quadratic-conic}
  x^2+y^2+r(x+y)+c=Hxy.
\end{equation}

\item On every nondegenerate segment, alternating the two coordinate jumps produces a sequence $(u_j)$ satisfying
\begin{equation}\label{eq:general-affine-rec}
  u_{j+1}=Hu_j-u_{j-1}-r.
\end{equation}
If $H\ne2$, then
\begin{equation}\label{eq:general-linearize}
  v_j\defeq(H-2)u_j-r
\end{equation}
satisfies the homogeneous Lucas-type recurrence
\begin{equation}\label{eq:general-hom-rec}
  v_{j+1}=Hv_j-v_{j-1}.
\end{equation}
For $H=2$, the affine recurrence instead has constant second difference
\[
  u_{j+1}-2u_j+u_{j-1}=-r.
\]

\item Every monic quadratic is scalar-self-reciprocal with respect to its constant term:
\begin{equation}\label{eq:monic-quadratic-scalar-reciprocal}
  \Rcal_cQ(T)=cQ(T).
\end{equation}
\end{enumerate}
\end{proposition}

\begin{proof}
Part~\textup{(i)} is the bridge from the divisibility-state formulation to a fixed conic.  The numerator in \eqref{eq:H-general} is congruent to $Q(y)$ modulo $x$ and to $Q(x)$ modulo $y$.  Under the hypotheses of part~\textup{(i)} it is therefore divisible by both $x$ and $y$, and coprimality makes it divisible by $xy$.

Now assume $H\defeq H_Q(x,y)\in\Z$.  Equation~\eqref{eq:general-quadratic-conic}, regarded as a quadratic in $x$, is
\[
  x^2+(r-Hy)x+Q(y)=0.
\]
The other root is therefore
\[
  Hy-r-x=\frac{Q(y)}x,
\]
which proves \eqref{eq:general-quadratic-jump}.  Replacing $x$ by this companion root leaves the same quadratic equation.  When $x'\ne0$, symmetry in the two variables gives \eqref{eq:H-invariance}, and rewriting the same conic as a quadratic in either coordinate gives the displayed divisibilities.

Parts~\textup{(ii)} and~\textup{(iii)} are the standard Vieta companion-root mechanism and its usual affine linear recurrence on the nondegenerate locus.  Applying the companion-root formula successively gives \eqref{eq:general-affine-rec}, and substitution of \eqref{eq:general-linearize} gives \eqref{eq:general-hom-rec}.  Finally,
\[
  \Rcal_cQ(T)
  =T^2Q(c/T)
  =c^2+rcT+cT^2
  =cQ(T),
\]
which proves the reciprocal identity \eqref{eq:monic-quadratic-scalar-reciprocal}.
\end{proof}

\begin{remark}[the quadratic boundary case]\label{rem:quadratic-boundary}
Parts~\textup{(ii)}--\textup{(iii)} of Proposition~\ref{prop:quadratic-vieta-bridge} are conventional Vieta jumping.  The bridge relevant here is that mutual divisibility yields an integral conic parameter, while
\[
  \Rcal_cQ=cQ
\]
shows why a single monic quadratic state collapses, in degree two, to the same polynomial up to a scalar.  Coprimality of the modulus with $c$ automatically permits this cancellation; Proposition~\ref{prop:content-obstruction} gives the exact criterion without that hypothesis.  Once the integral conic parameter is known, the Vieta identity gives closure directly.  Unequal quadratic states can still change literally under reciprocity, but Proposition~\ref{prop:quadratic-gauge} below shows that their entire coefficient dynamics is rationally conjugate to fixed-conic Vieta jumping.  In higher degree no comparable reduction to fixed-conic Vieta dynamics is available in general, although Proposition~\ref{prop:four-step-gauge-closure} gives a general four-step rational coefficient gauge return for every positive-degree pair with nonzero constant terms.
\end{remark}

\subsection{Rational gauge reduction of changing quadratic states}\label{subsec:quadratic-gauge-arithmetic}

The need for such a gauge is already visible on a nonmonic fixed conic.  For
\[
  ap^2+aq^2+c(p+q)+e=kpq,
\]
the Vieta companion of $p$ is
\[
  p'=\frac{kq-c}{a}-p,
\]
which need not be integral when $a>1$.  After the rational rescaling
$P\defeq ap$, $Q\defeq aq$, the same equation becomes
\[
  P^2+Q^2+c(P+Q)+ae=\frac{k}{a}PQ,
\]
a monic conic over $\mathbb Q$.  Proposition~\ref{prop:quadratic-gauge} is the coefficient-dynamical version of this elementary rescaling.

\begin{proposition}[rational gauge reduction]\label{prop:quadratic-gauge}Let
\[
  P_0(T)=aT^2+bT+c,
  \qquad
  P_1(T)=dT^2+eT+f,
  \qquad acdf\ne0,
\]
and let $(P_n)_{n\in\Z}$ be a projective quadratic coefficient orbit, meaning
\begin{equation}\label{eq:projective-quadratic-orbit}
  P_{n+1}\ \propto\ \Rcal_{P_n(0)}P_{n-1}
  \qquad(n\in\Z),
\end{equation}
where $\propto$ denotes multiplication by a nonzero rational scalar.  Define
\begin{equation}\label{eq:gauge-AB}
  \widetilde A(T)\defeq\frac{f}{cd}P_0(T),
  \qquad
  \widetilde B(T)\defeq\frac{c}{af}P_1(T).
\end{equation}
Writing
\[
  \widetilde A(T)=a_0T^2+b_0T+c_0,
  \qquad
  \widetilde B(T)=d_0T^2+e_0T+f_0,
\]
one has
\begin{equation}\label{eq:gauge-relations}
  a_0d_0=1,
  \qquad
  f_0=d_0c_0,
  \qquad
  \Rcal_{c_0}\widetilde B=c_0\widetilde B,
  \qquad
  \Rcal_{f_0}\widetilde A=f_0\widetilde A.
\end{equation}
Moreover, there exist nonzero rational numbers $\lambda_n,\sigma_n$ such that
\begin{equation}\label{eq:gauge-polynomials}
  P_n(\lambda_nT)=\sigma_n
  \begin{cases}
    \widetilde A(T),&n\text{ even},\\
    \widetilde B(T),&n\text{ odd},
  \end{cases}
  \qquad
  \lambda_{n-1}\lambda_{n+1}=\sigma_n.
\end{equation}
Consequently every nondegenerate coordinate orbit with $x_n\ne0$ for all $n$,
\begin{equation}\label{eq:changing-quadratic-recurrence}
  x_{n-1}x_{n+1}=P_n(x_n),
\end{equation}
is carried by $x_n=\lambda_nu_n$ to the fixed alternating system
\begin{equation}\label{eq:fixed-alternating-quadratic}
  u_{n-1}u_{n+1}=
  \begin{cases}
    \widetilde A(u_n),&n\text{ even},\\
    \widetilde B(u_n),&n\text{ odd}.
  \end{cases}
\end{equation}
For an adjacent even--odd pair $(u,v)$, necessarily with $uv\ne0$, this is ordinary Vieta jumping on the fixed conic
\begin{equation}\label{eq:gauge-conic}
  u^2+d_0v^2+d_0b_0u+e_0v+f_0=Huv,
\end{equation}
where
\begin{equation}\label{eq:gauge-H}
  H\defeq
  \frac{u^2+d_0v^2+d_0b_0u+e_0v+f_0}{uv}
\end{equation}
is preserved by both root replacements.
\end{proposition}

\begin{proof}
The relations in \eqref{eq:gauge-relations} follow by direct substitution.  In particular, the reciprocal transform of either fixed polynomial, with the reciprocal parameter supplied by the opposite parity, is a scalar multiple of itself.

Set $\lambda_0=\lambda_1=1$ and choose $\sigma_0,\sigma_1$ so that \eqref{eq:gauge-polynomials} holds for $n=0,1$.  Suppose it holds at $n-1,n$.  Choose
\[
  \lambda_{n+1}\defeq\frac{\sigma_n}{\lambda_{n-1}}.
\]
Substituting the two induction hypotheses into
$\Rcal_{P_n(0)}P_{n-1}$ and using \eqref{eq:gauge-relations} shows that
$P_{n+1}(\lambda_{n+1}T)$ is proportional to the fixed polynomial of parity $n+1$; absorb that proportionality into $\sigma_{n+1}$.  This proves \eqref{eq:gauge-polynomials} for positive indices.  Since $\Rcal_c^2P=c^2P$ for quadratics, the projective recurrence \eqref{eq:projective-quadratic-orbit} is reversible; applying the same induction to the reversed orbit gives the negative indices.  Substitution $x_n=\lambda_nu_n$ in \eqref{eq:changing-quadratic-recurrence} then gives \eqref{eq:fixed-alternating-quadratic}.

Finally, regarding \eqref{eq:gauge-conic} as a quadratic in $u$, the product of its roots is
\[
  d_0v^2+e_0v+f_0=\widetilde B(v).
\]
Regarding it as a quadratic in $v$, the product of the roots is
\[
  \frac{u^2+d_0b_0u+f_0}{d_0}
  =\widetilde A(u),
\]
by \eqref{eq:gauge-relations}.  Thus the two quotient jumps in \eqref{eq:fixed-alternating-quadratic} are precisely the two Vieta reflections of \eqref{eq:gauge-conic}, and $H$ is invariant.
\end{proof}

\begin{remark}[Lucas and Laurent viewpoints]\label{rem:gauge-lp}
Two successive Vieta reflections of \eqref{eq:gauge-conic} have determinant $1$ and trace
\[
  K\defeq\frac{H^2}{d_0}-2.
\]
Away from the parabolic case $H^2=4d_0$, an affine translation therefore makes each parity subsequence satisfy
$w_{j+2}=Kw_{j+1}-w_j$.  Thus the changing quadratic coefficient states are rationally conjugate to fixed-conic Lucas dynamics; the remaining arithmetic issue is integrality of the gauge and of primitive normalization.

Lam--Pylyavskyy's rank-two Laurent-phenomenon mutation has a related reciprocal-substitution and factor-removal form~\cite[Sec.~2 and Thm.~5.1]{LamPylyavskyy2016}.  We use no symbolic Laurentness to infer integral specialization at nonunit integers; that issue is measured here by Proposition~\ref{prop:content-obstruction}.
\end{remark}

\subsection{Arithmetic realization of recurrent quadratic states}\label{subsec:quadratic-recurrent}

For a nonzero integer polynomial $P$, let $\operatorname{pp}(P)$ denote its primitive part, with sign chosen so that its leading coefficient is positive.  On primitive quadratics with nonzero constant terms define the normalized jump-and-swap map
\begin{equation}\label{eq:J-map}
  \mathcal J(A,B)
  \defeq
  \bigl(\operatorname{pp}(\Rcal_{A(0)}B),A\bigr).
\end{equation}
This notation records the coefficient transformation needed below; the intrinsic classification of its integral coefficient orbits is separate from the present arithmetic problem.

\begin{proposition}[finite obstruction primes and an $S$-coprime chamber]\label{prop:S-coprime-chamber}
Let $P_0=aT^2+bT+c$ and $P_1=dT^2+eT+f$ be primitive quadratics with $acdf\ne0$, and extend them in both directions by the normalized coefficient map \eqref{eq:J-map}.  Let
\begin{equation}\label{eq:obstruction-set}
  S\defeq\{p\text{ prime}:p\mid acdf\}.
\end{equation}
Every prime dividing a normalization content anywhere on this coefficient orbit belongs to $S$.

Suppose, in addition, that for every coefficient polynomial $P_n$ on the orbit and every $p\in S$,
\begin{equation}\label{eq:no-unit-root}
  P_n(t)\not\equiv0\pmod p
  \qquad(t\in\mathbb F_p^{\times}).
\end{equation}
Then any valid arithmetic state on this orbit whose two coordinates are $p$-adic units for every $p\in S$ can be continued indefinitely in both directions through the primitive normalized coefficient states; all subsequent coordinates remain coprime to every prime in $S$.
\end{proposition}

\begin{proof}
For a quadratic $B=dT^2+eT+f$, the raw reciprocal transform with parameter $c$ has endpoints $f$ and $dc^2$.  After division by its content, every prime dividing either new endpoint already divides one of $d,f,c$.  Induction proves the first assertion.

Now assume the two current coordinates are $p$-adic units for every $p\in S$.  The quotient polynomial evaluated at the retained coordinate is a $p$-adic unit for each $p\in S$ by \eqref{eq:no-unit-root}; hence the next integral quotient is again coprime to every prime in $S$.  Proposition~\ref{prop:content-obstruction} shows that every scalar removed in primitive normalization is supported on $S$, while the retained modulus is coprime to every prime in $S$, so the scalar can be cancelled.  The same argument applies in either direction.
\end{proof}

\begin{proposition}[a recurrent nonunit four-cycle]\label{prop:recurrent-nonunit-cycle}
Let $c\ge2$ and put
\[
\begin{aligned}
  P_0(T)&\defeq T^2+cT+c,&
  P_1(T)&\defeq T^2-cT-c,\\
  P_2(T)&\defeq T^2-cT+c,&
  P_3(T)&\defeq T^2+cT-c.
\end{aligned}
\]
Then primitive reciprocal normalization gives the four-cycle
\begin{equation}\label{eq:nonunit-four-cycle}
  \cS(P_0,P_3)
  \longrightarrow\cS(P_1,P_0)
  \longrightarrow\cS(P_2,P_1)
  \longrightarrow\cS(P_3,P_2)
  \longrightarrow\cS(P_0,P_3),
\end{equation}
and every raw reciprocal transform in this cycle has content $c$.
Moreover, the original state $\cS(P_0,P_3)$ has infinitely many positive integral points on which the cycle can be repeated indefinitely.
\end{proposition}

\begin{proof}
The coefficient cycle follows from
\begin{equation}\label{eq:nonunit-four-reciprocals}
\begin{aligned}
  \Rcal_cP_3&=-cP_1,&
  \Rcal_{-c}P_0&=cP_2,\\
  \Rcal_cP_1&=-cP_3,&
  \Rcal_{-c}P_2&=cP_0.
\end{aligned}
\end{equation}
Thus the scalar removed at each primitive normalization is genuinely nonunit.

Consider the conic
\begin{equation}\label{eq:nonunit-four-conic}
  x^2-y^2+c(x-y)+c=cxy.
\end{equation}
It contains $(1,1)$, and rearranging \eqref{eq:nonunit-four-conic} gives
\[
  P_0(x)=y(cx+y+c),
  \qquad
  P_3(y)=x(x+c-cy),
\]
so every nonzero integral point on this conic lies in $\cS(P_0,P_3)$.  For such a point define
\begin{equation}\label{eq:nonunit-four-quotients}
\begin{aligned}
  z_1&\defeq cx+y+c,&
  z_2&\defeq cz_1+x+c,\\
  z_3&\defeq cz_2+z_1-c,&
  z_4&\defeq cz_3+z_2-c.
\end{aligned}
\end{equation}
Using \eqref{eq:nonunit-four-conic} successively gives
\[
  P_0(x)=yz_1,
  \qquad P_1(z_1)=xz_2,
  \qquad P_2(z_2)=z_1z_3,
  \qquad P_3(z_3)=z_2z_4.
\]
Hence these are precisely the four quotient coordinates in \eqref{eq:nonunit-four-cycle}.  Direct substitution shows that $(z_4,z_3)$ again satisfies \eqref{eq:nonunit-four-conic}; explicitly, one full cycle is
\begin{equation}\label{eq:nonunit-four-affine-map}
\binom{x}{y}
\longmapsto
\begin{pmatrix}
 c^4+3c^2+1 & c^3+2c\\
 c^3+2c & c^2+1
\end{pmatrix}
\binom{x}{y}
+
\binom{c^2(c^2+c+1)}{c^2(c+1)}.
\end{equation}
If $x,y>0$ and $x\equiv y\equiv1\pmod c$, then \eqref{eq:nonunit-four-quotients} gives $z_i>0$ and $z_i\equiv1\pmod c$ for all $i$.  Since $P_i(T)\equiv T^2\pmod p$ for every prime $p\mid c$, Proposition~\ref{prop:S-coprime-chamber} permits the cancellation of the content $c$ at every step.  Starting from $(1,1)$, conic preservation and the preceding congruences therefore allow the cycle to be repeated indefinitely.  The first coordinate in \eqref{eq:nonunit-four-affine-map} strictly increases at every full cycle, so its iterates give infinitely many distinct positive integral points of $\cS(P_0,P_3)$.
\end{proof}

\begin{proposition}[Lucas--Pell linearization of the recurrent four-cycle]\label{prop:lucas-pell-cycle}
For $c\ge2$, put
\[
  D\defeq c^2+4,
  \qquad
  Q_c\defeq
  \begin{pmatrix}c&1\\1&0\end{pmatrix},
\]
and define the companion Lucas sequences by
\[
  U_0\defeq0,\quad U_1\defeq1,\quad U_{n+2}\defeq cU_{n+1}+U_n,
\]
\[
  V_0\defeq2,\quad V_1\defeq c,\quad V_{n+2}\defeq cV_{n+1}+V_n.
\]
Then the linear part $M_c$ of the full-cycle map \eqref{eq:nonunit-four-affine-map} satisfies
\begin{equation}\label{eq:Mc-Lucas}
  M_c=Q_c^4=
  \begin{pmatrix}
    U_5&U_4\\
    U_4&U_3
  \end{pmatrix},
  \qquad
  \operatorname{tr} M_c=V_4=c^4+4c^2+2,
  \qquad
  \det M_c=1.
\end{equation}

Let $u_0\defeq y$, $u_1\defeq x$, and thereafter list consecutively the quotient coordinates generated by repeated traversal of \eqref{eq:nonunit-four-cycle}.  Then
\begin{equation}\label{eq:general-forced-Lucas}
  u_{n+2}=cu_{n+1}+u_n+c\varepsilon_n,
  \qquad
  \varepsilon_n\defeq(-1)^{\lfloor n/2\rfloor},
\end{equation}
so
\begin{equation}\label{eq:general-fourth-order}
  u_{n+4}=cu_{n+3}+cu_{n+1}+u_n,
\end{equation}
whose characteristic polynomial factors as
\begin{equation}\label{eq:general-fourth-factor}
  X^4-cX^3-cX-1=(X^2+1)(X^2-cX-1).
\end{equation}
More explicitly, define the four-periodic rational correction $(p_n)$ by
\begin{equation}\label{eq:periodic-correction}
  p_0\defeq\frac{c(c-2)}{D},
  \qquad
  p_1\defeq-\frac{c(c+2)}{D},
  \qquad
  p_{n+2}\defeq-p_n,
\end{equation}
and put $h_n\defeq u_n-p_n$.  Then
\begin{equation}\label{eq:homogeneous-Lucas}
  h_{n+2}=ch_{n+1}+h_n.
\end{equation}

Finally, for a point $(x,y)$ on the conic \eqref{eq:nonunit-four-conic}, set
\begin{equation}\label{eq:Pell-coordinates-general-c}
  t\defeq2x-cy+c,
  \qquad
  Y\defeq Dy-c(c-2).
\end{equation}
Then \eqref{eq:nonunit-four-conic} is equivalent to the generalized Pell equation
\begin{equation}\label{eq:general-c-Pell}
  Y^2-Dt^2=16c.
\end{equation}
Writing $(x',y')$ for one full-cycle image and defining $t',Y'$ by \eqref{eq:Pell-coordinates-general-c}, one has
\begin{equation}\label{eq:general-c-unit-action}
  Y'+t'\sqrt D
  =\alpha^4\bigl(Y+t\sqrt D\bigr),
  \qquad
  \alpha\defeq\frac{c+\sqrt D}{2}.
\end{equation}
Equivalently,
\begin{equation}\label{eq:general-c-Pell-matrix}
  \binom{Y'}{t'}
  =\frac12
  \begin{pmatrix}
    V_4&D U_4\\
    U_4&V_4
  \end{pmatrix}
  \binom{Y}{t}.
\end{equation}
Consequently, along full-cycle iterates, either centered coordinate $W_k=Y_k$ or $W_k=t_k$ satisfies
\begin{equation}\label{eq:full-cycle-Lucas-recurrence}
  W_{k+2}=V_4W_{k+1}-W_k.
\end{equation}
Thus the Pell phenomenon in the recurrent nonunit family is present for every $c$, not only for the specialization $c=2$.
\end{proposition}

\begin{proof}
The standard matrix identity
\[
  Q_c^n=
  \begin{pmatrix}
    U_{n+1}&U_n\\
    U_n&U_{n-1}
  \end{pmatrix}
  \qquad(n\ge1)
\]
follows by induction from the recurrence for $U_n$.  Taking $n=4$ gives the matrix in \eqref{eq:nonunit-four-affine-map}; also $U_5+U_3=V_4$ and $\det Q_c=-1$, proving \eqref{eq:Mc-Lucas}.

The four quotient formulas \eqref{eq:nonunit-four-quotients}, followed by the same four formulas at the next returned state, give \eqref{eq:general-forced-Lucas}.  Since $\varepsilon_{n+2}=-\varepsilon_n$, adding the instances of \eqref{eq:general-forced-Lucas} at $n$ and $n+2$ eliminates the periodic term and yields \eqref{eq:general-fourth-order}; \eqref{eq:general-fourth-factor} is immediate.  Direct substitution of \eqref{eq:periodic-correction} into \eqref{eq:general-forced-Lucas} shows that $(p_n)$ has exactly the same inhomogeneous term, so subtraction gives \eqref{eq:homogeneous-Lucas}.

For the Pell coordinates, direct expansion gives the identity
\[
  Y^2-Dt^2-16c
  =-4D\bigl(x^2-y^2+c(x-y)+c-cxy\bigr),
\]
which proves the equivalence with \eqref{eq:general-c-Pell}.  Substitution of the affine map \eqref{eq:nonunit-four-affine-map} into \eqref{eq:Pell-coordinates-general-c} gives
\[
  2Y'=V_4Y+D U_4t,
  \qquad
  2t'=U_4Y+V_4t.
\]
If $\beta\defeq(c-\sqrt D)/2$, then $\alpha\beta=-1$ and the usual Binet identities give
\[
  \alpha^n=\frac{V_n+U_n\sqrt D}{2}.
\]
The preceding two equations are therefore exactly \eqref{eq:general-c-unit-action}.  Since $\operatorname{Norm}(\alpha^4)=1$, preservation of \eqref{eq:general-c-Pell} is also transparent in these coordinates.
\end{proof}

\begin{remark}[OEIS and classical Lucas specializations]\label{rem:oeis-lucas-family}
The arrays OEIS A352361 and A352362 collect, respectively, the Fibonacci- and Lucas-polynomial families represented here by $U_n(c)$ and $V_n(c)$ as $c$ varies~\cite{OEIS}.  Thus the Pell row $c=2$ is one visible specialization of the general structure above.  The same norm-one description also gives the Chebyshev identity
\[
  V_{4k}(c)=2T_k\!\left(\frac{V_4(c)}2\right),
\]
where $T_k$ is the Chebyshev polynomial of the first kind.  Another particularly classical case is $c=4$: then
\[
  \alpha=2+\sqrt5=\varphi^3,
  \qquad
  V_4=322=L_{12},
\]
where $\varphi=(1+\sqrt5)/2$ and $L_n$ denotes the ordinary Lucas numbers.  Consequently $V_{4k}(4)=L_{12k}$, the sequence OEIS A089775, and the corresponding full-cycle second-order recurrence has signature $(322,-1)$~\cite{OEIS}.  These identifications come from the matrix and norm-one-unit structure, rather than from matching isolated initial terms.
\end{remark}

For $c=2$ the successive quotient coordinates begin
\[
  1,1,5,13,29,69,169,409,985,2377,5741,13861,\ldots,
\]
and the corresponding solutions of
\[
  y\mid x^2+2x+2,
  \qquad
  x\mid y^2+2y-2
\]
begin
\[
  (1,1),\qquad(69,29),\qquad(2377,985),\qquad(80781,33461),\ldots.
\]
Writing these successive quotient coordinates as $(u_n)_{n\ge0}$, Proposition~\ref{prop:lucas-pell-cycle} specializes to
\begin{equation}\label{eq:c2-fourth-order}
  u_{n+4}=2u_{n+3}+2u_{n+1}+u_n,
\end{equation}
with characteristic factorization
\[
  X^4-2X^3-2X-1=(X^2+1)(X^2-2X-1).
\]
Let $\operatorname{Pell}(n)$ denote the Pell numbers, defined by
$\operatorname{Pell}(0)\defeq0$, $\operatorname{Pell}(1)\defeq1$, and
$\operatorname{Pell}(n+2)=2\operatorname{Pell}(n+1)+\operatorname{Pell}(n)$.
For $c=2$, the correction \eqref{eq:periodic-correction} is
$0,-1,0,1,\ldots$; hence $h_0=1$, $h_1=2$, and
$h_n=\operatorname{Pell}(n+1)$.  Therefore
\begin{equation}\label{eq:c2-pell-identities}
  u_{2n}=\operatorname{Pell}(2n+1),
  \qquad
  u_{2n+1}=\operatorname{Pell}(2n+2)-(-1)^n
  \qquad(n\ge0).
\end{equation}
Thus the even and odd subsequences are respectively
\[
  1,5,29,169,985,\ldots\quad\text{(OEIS A001653)},
  \qquad
  1,13,69,409,2377,\ldots\quad\text{(OEIS A105058)}.
\]
These identifications are recorded in the OEIS~\cite{OEIS}.  In particular, for the full-cycle points $(x_k,y_k)=(u_{4k+1},u_{4k})$,
\[
  y_k=\operatorname{Pell}(4k+1),
  \qquad
  x_k=\operatorname{Pell}(4k+2)-1.
\]
The sequence $(y_k)=1,29,985,33461,1136689,\ldots$ is OEIS A316708 and satisfies
$y_{k+2}=34y_{k+1}-y_k$~\cite{OEIS}.  Unlike the point-dependent monic normalization in Section~\ref{sec:nonunit}, this example uses one recurrent primitive coefficient cycle; no point-dependent CRT parameter is chosen.

For the point-dependent normalization below, we shall use
\begin{equation}\label{eq:Qa}
  Q_a(T)\defeq T^2-T+a.
\end{equation}

\section{A nonunit quadratic divisibility system}\label{sec:nonunit}

At the fixed-equation level, Mills's 1956 divisibility hypotheses already provide a classical way to handle nonunit quadratic leading coefficients: they ensure that the Vieta companion-root sums are integral globally for the whole equation \cite{Mills1956}.  The normalization below is of a different kind: it is point-dependent and takes place at the level of congruences rather than by imposing global coefficient divisibility.

We now consider the motivating quadratic system
\begin{equation}\label{eq:two-div}
  m\mid 2n^2-2n+1,
  \qquad
  n\mid 2m^2-2m+1,
\end{equation}
for nonzero integers $m,n$.  Write
\begin{equation}\label{eq:F-def}
  F(T)\defeq2T^2-2T+1.
\end{equation}
Since $F(t)$ is odd for every integer $t$, every nonzero divisor of $F(t)$ is odd; hence every solution of \eqref{eq:two-div} has $m,n$ odd.  Every common divisor of $m$ and $n$ divides $F(n)\equiv1$ modulo that divisor, so
\begin{equation}\label{eq:mn-coprime}
  \gcd(m,n)=1.
\end{equation}
The pair of divisibilities is therefore equivalent to a single equation
\begin{equation}\label{eq:k-quadratic}
  2m^2+2n^2-2m-2n+1=kmn
\end{equation}
for an odd integer $k$.

A standard symmetric change of variables gives a Pell equation.  Set
\begin{equation}\label{eq:u-v}
  m\defeq u+v,
  \qquad n\defeq u-v.
\end{equation}
Then \eqref{eq:k-quadratic} becomes
\[
  (k-4)u^2+4u-1-(k+4)v^2=0.
\]
With
\begin{equation}\label{eq:t-def}
  t\defeq(k-4)u+2,
\end{equation}
we obtain
\begin{equation}\label{eq:two-div-Pell}
  t^2-(k^2-16)v^2=k,
\end{equation}
subject to the lattice condition
\begin{equation}\label{eq:lattice-condition}
  t\equiv2\pmod{k-4}.
\end{equation}
This is the fixed-equation viewpoint.

\subsection{Persistence along a Pell orbit}

The lattice condition \eqref{eq:lattice-condition} is automatically persistent along every second step of a Pell orbit.

\begin{proposition}[Pell-orbit persistence]\label{prop:Pell-persistence}
Let $k>5$ be odd, set $D\defeq k^2-16$, and suppose that
\[
  t_0^2-Dv_0^2=k,
  \qquad
  t_0\equiv2\pmod{k-4}.
\]
Let $\varepsilon_t,\varepsilon_v\in\Z$ satisfy $\varepsilon_t+\varepsilon_v\sqrt D>1$ and $\varepsilon_t^2-D\varepsilon_v^2=1$, and define
\[
  t_s+v_s\sqrt D
  \defeq
  (t_0+v_0\sqrt D)
  (\varepsilon_t+\varepsilon_v\sqrt D)^s.
\]
Then
\[
  t_{2r}\equiv2\pmod{k-4}
  \qquad(r\ge0).
\]
If $(t_0,v_0)$ comes from a positive solution of \eqref{eq:k-quadratic}, then infinitely many of the even iterates also come from positive solutions.
\end{proposition}

\begin{proof}
Since $D=(k-4)(k+4)$,
\[
  \varepsilon_t^2-D\varepsilon_v^2=1
  \quad\Longrightarrow\quad
  \varepsilon_t^2\equiv1\pmod{k-4}.
\]
Reducing the product formula modulo $k-4$ gives
\[
  t_s\equiv t_0\varepsilon_t^s\pmod{k-4},
\]
and hence $t_{2r}\equiv t_0\equiv2\pmod{k-4}$.

For positivity, interchange $m,n$ if necessary and take $v_0>0$.  Along the positive Pell orbit, $t_{2r},v_{2r}$ tend to infinity and
\[
  \frac{t_{2r}}{v_{2r}}\longrightarrow\sqrt D.
\]
The inverse transformation is
\[
  u_{2r}=\frac{t_{2r}-2}{k-4},
  \qquad
  m_{2r}=u_{2r}+v_{2r},
  \qquad
  n_{2r}=u_{2r}-v_{2r}.
\]
Since
\[
  \frac{\sqrt D}{k-4}
  =\sqrt{\frac{k+4}{k-4}}>1,
\]
we have $u_{2r}>v_{2r}>0$ for all sufficiently large $r$.  Thus infinitely many even iterates yield positive $(m_{2r},n_{2r})$.
\end{proof}

\begin{remark}[Vieta jumps as unit multiplication]\label{rem:vieta-unit-multiplication}
The fixed-conic and Pell descriptions are the same dynamics in different coordinates.
Let $D>0$ be nonsquare, let
\[
  \eta=a+b\sqrt D,
  \qquad a,b\in\Z,
  \qquad a^2-Db^2=\delta\in\{\pm1\},
\]
and write
\[
  U=x+ay,
  \qquad V=by.
\]
Then $U^2-DV^2=N$ is equivalent to
\begin{equation}\label{eq:vieta-unit-conic}
  x^2+2axy+\delta y^2=N.
\end{equation}
Let $R_x$ and $R_y$ denote the two Vieta companion-root involutions of
\eqref{eq:vieta-unit-conic}:
\[
  R_x(x,y)=(-x-2ay,y),
  \qquad
  R_y(x,y)=(x,-y-2a\delta x).
\]
If $(x',y')=R_yR_x(x,y)$ and
$U'=x'+ay'$, $V'=by'$, then direct substitution gives
\begin{equation}\label{eq:vieta-unit-multiplication}
  U'+V'\sqrt D
  =\delta\eta^2\bigl(U+V\sqrt D\bigr).
\end{equation}
Since $\operatorname{Norm}_{\Q(\sqrt D)/\Q}(\delta\eta^2)=1$, two
successive Vieta jumps amount to multiplication by a norm-one unit.  Thus a
fixed-conic Vieta orbit and the corresponding Pell unit orbit are the same
orbit after the displayed coordinate change.  This identifies the dynamics
once an admissible seed exists; existence of such an integral seed, and the
possibility of several unit orbits, are separate arithmetic questions.  The
explicit unit action in \eqref{eq:general-c-unit-action} is a recurrent
nonunit instance of the same mechanism.
\end{remark}

\subsection{The UVDose--Wagemakers family and fixed-parameter Pell orbits}

An unpublished 2016 note by the forum user UVDose gives an explicit changing-polynomial construction of positive solutions of the equivalent system
\[
  x+1\mid y^2+1,
  \qquad
  y+1\mid x^2+1
\]
\cite{UVDose2016}.  Under $x=2m-1$ and $y=2n-1$ this is exactly \eqref{eq:two-div}.  A later research note by Gerard Wagemakers develops the same problem further: it gives a Pell/continued-fraction classification by the parameter $k$, records the persistence of the admissibility congruences along a fixed parity of a Pell orbit, and derives an explicit formula for the parameters arising from the UVDose construction \cite{Wagemakers2019}.  We record the construction here because it is a useful earlier example of changing-polynomial divisibility dynamics.  The arguments below are self-contained: the two informal notes are cited for historical provenance, while every statement used later is either reproved here or reduced to an explicit calculation.

For $N\ge1$, set
\begin{equation}\label{eq:uvdose-P}
  P_r(X)\defeq X^2-2^{r+1}X+2^{2r+1}
  =(X-2^r)^2+2^{2r},
\end{equation}
and define
\begin{equation}\label{eq:uvdose-gamma}
  \gamma_0\defeq\gamma_1\defeq1,
  \qquad
  \gamma_{j+2}\defeq
  \frac{P_{N-j}(\gamma_{j+1})}{\gamma_j}
  \quad(0\le j\le N).
\end{equation}
UVDose proves that all quantities in \eqref{eq:uvdose-gamma} are positive odd integers.  With
\begin{equation}\label{eq:uvdose-beta}
  \beta_N\defeq
  \frac{2\gamma_{N+2}^2-2\gamma_{N+2}+1}{\gamma_{N+1}},
\end{equation}
and
\begin{equation}\label{eq:uvdose-mn}
  m_N\defeq\gamma_{N+2},
  \qquad
  n_N\defeq\beta_N,
\end{equation}
his note proves that $(m_N,n_N)$ is a positive solution of \eqref{eq:two-div}; equivalently, $(2m_N-1,2n_N-1)$ solves the original shifted divisibility problem \cite{UVDose2016}.  The same note proves the strict growth estimate
\[
  2m_N-1>2^{N+1}-1.
\]
The local reciprocal theorem makes the changing-polynomial mechanism behind this construction explicit.

\begin{proposition}[the UVDose recursion as a congruence-jump ladder]\label{prop:uvdose-ladder}
For $0\le j\le N$, the consecutive terms in \eqref{eq:uvdose-gamma} satisfy
\begin{equation}\label{eq:uvdose-state}
  \gamma_j\mid P_{N-j}(\gamma_{j+1}),
  \qquad
  \gamma_{j+1}\mid P_{N-j+1}(\gamma_j).
\end{equation}
Each passage $(\gamma_j,\gamma_{j+1})\mapsto(\gamma_{j+1},\gamma_{j+2})$ is therefore a quotient jump between two companion quadratic divisibility states.  Up to interchanging the two coordinates after each jump, the polynomial states descend through
\begin{equation}\label{eq:uvdose-state-ladder}
  (P_N,P_{N+1})\rightsquigarrow(P_{N-1},P_N)
  \rightsquigarrow\cdots\rightsquigarrow(P_0,P_1)
  \rightsquigarrow(F,P_0)\rightsquigarrow(F,F).
\end{equation}
The last state is attained at $(m_N,n_N)$.  Hence the UVDose construction is a family, indexed by $N$, of finite changing-state congruence-jump ladders ending in the original nonunit system.
\end{proposition}

\begin{proof}
Write
\[
  c_r\defeq P_r(0)=2^{2r+1}.
\]
A direct calculation gives, for $r\ge1$,
\begin{equation}\label{eq:uvdose-reciprocal-step}
  \Rcal_{c_r}P_{r+1}(T)
  =2^{2r+3}P_{r-1}(T),
\end{equation}
and at the bottom of the ladder
\begin{equation}\label{eq:uvdose-boundary-reciprocals}
  \Rcal_2P_1(T)=4F(T),
  \qquad
  \Rcal_1P_0(T)=F(T).
\end{equation}
The two divisibilities in \eqref{eq:uvdose-state} are trivial at $j=0$, since $\gamma_0=\gamma_1=1$.  Suppose that they hold for some $j<N$, and put $r\defeq N-j$.  By definition,
\[
  \gamma_{j+2}=\frac{P_r(\gamma_{j+1})}{\gamma_j}.
\]
Apply Theorem~\ref{thm:local-reciprocal} with
\[
  x\defeq\gamma_{j+1},\qquad
  y\defeq\gamma_j,\qquad
  A\defeq P_r,\qquad
  B\defeq P_{r+1}.
\]
It gives
\[
  \gamma_{j+1}\mid
  \Rcal_{c_r}P_{r+1}(\gamma_{j+2})
  =2^{2r+3}P_{r-1}(\gamma_{j+2}).
\]
Since $\gamma_{j+1}$ is odd, the power of $2$ can be cancelled.  Together with
$\gamma_{j+2}\mid P_r(\gamma_{j+1})$, this is exactly \eqref{eq:uvdose-state} with $j$ replaced by $j+1$.

At $j=N$, the same argument uses the first identity in \eqref{eq:uvdose-boundary-reciprocals}: from
\[
  \gamma_N\mid P_0(\gamma_{N+1}),
  \qquad
  \gamma_{N+1}\mid P_1(\gamma_N)
\]
and
\[
  \gamma_{N+2}=\frac{P_0(\gamma_{N+1})}{\gamma_N}
\]
we obtain, after cancelling the factor $4$ modulo the odd integer $\gamma_{N+1}$,
\[
  \gamma_{N+1}\mid F(\gamma_{N+2}).
\]
Thus $\beta_N$ in \eqref{eq:uvdose-beta} is the next quotient.  Moreover
$\gamma_{N+2}\mid P_0(\gamma_{N+1})$ by construction.  A final application of Theorem~\ref{thm:local-reciprocal}, now with $A=F$ and $B=P_0$, uses the second identity in \eqref{eq:uvdose-boundary-reciprocals} to give
\[
  \gamma_{N+2}\mid F(\beta_N).
\]
Hence $(m_N,n_N)=(\gamma_{N+2},\beta_N)$ lies in the state $(F,F)$, proving the stated ladder interpretation.
\end{proof}

Because the normalized quadratic map $\mathcal J$ is bijective, the coefficient ladders in \eqref{eq:uvdose-state-ladder} are nested segments of one bi-infinite coefficient orbit.  Indeed,
\[
  \mathcal J(P_r,P_{r+1})=(P_{r-1},P_r)\quad(r\ge1),
\]
while
\[
  \mathcal J(P_0,P_1)=(F,P_0),
  \qquad
  \mathcal J(F,P_0)=(F,F).
\]
Continuing through $(F,F)$ produces the same states in the reverse orientation.  Thus the auxiliary polynomial family is forced by normalized reciprocal dynamics once the central state $(F,F)$ is fixed.

\begin{corollary}[exact parity locus]\label{cor:uvdose-parity}
Fix $N\ge1$ and let $(x,y)\in\cS(P_N,P_{N+1})$.  The point $(x,y)$ can be transported through every primitive normalized state of the bi-infinite coefficient orbit above if and only if $x$ and $y$ are odd.
\end{corollary}

\begin{proof}
All normalization contents on this orbit are powers of $2$.  Moreover
\[
  P_r(t)\equiv1\pmod2
  \quad(t\text{ odd}),
  \qquad
  F(t)\equiv1\pmod2.
\]
Hence odd coordinates remain odd under every integral quotient jump, and Proposition~\ref{prop:S-coprime-chamber} permits every primitive normalization.  This proves sufficiency.

Conversely, the forward coefficient ladder reaches $(F,F)$.  At any valid point $(u,v)\in\cS(F,F)$ both $u$ and $v$ are odd because $F(t)$ is odd for every integer $t$.  Reversing the quotient relations through $F,P_0,P_1,\ldots$ preserves oddness, since each of these polynomials takes odd values at odd arguments.  Thus the original coordinates $x,y$ must be odd.
\end{proof}

For the first two values of $N$, the quotient chains themselves are already illustrative:
\begin{align*}
N=1:\quad
 &(1,1)\xrightarrow{P_1}(1,5)
 \xrightarrow{P_0}(5,17)
 \xrightarrow{F}(17,109),
 &&k_1=13,\\
N=2:\quad
 &(1,1)\xrightarrow{P_2}(1,25)
 \xrightarrow{P_1}(25,533)
 \xrightarrow{P_0}(533,11321)
 \xrightarrow{F}(11321,480877),
 &&k_2=85.
\end{align*}
Here an arrow labelled by $Q$ means that the new coordinate is the quotient of $Q$ evaluated at the retained coordinate by the coordinate being replaced.  The numerical chains for different $N$ are not initial segments of one arithmetic point orbit: each starts from the unit seed in a different coefficient state.  Their coefficient-state ladders, however, are nested segments of the single bi-infinite normalized orbit described above.  Their unbounded lengths and distinct endpoints therefore give an explicit arithmetic realization of arbitrarily long segments of that coefficient dynamics.

Wagemakers subsequently obtained the closed parameter formula
\begin{equation}\label{eq:kN}
  k_N=2^{2N+3}-3\cdot2^{N+2}+5
  =13,85,421,1861,\ldots
\end{equation}
for the endpoints above~\cite[Prop.~3.2(6)]{Wagemakers2019}.  In the present notation its meaning is transparent after the rescaling
$T\defeq2^N$ and $g_j\defeq2^j\gamma_{j+1}$, for which
\[
  g_{j-1}g_{j+1}=g_j^2-2Tg_j+2T^2.
\]
Hence
\[
  I(a,b)\defeq
  \frac{a^2+b^2-2T(a+b)+2T^2}{ab}
\]
is constant along the ladder.  Evaluating $I$ at its two ends gives
$I=k_N/2$, recovering \eqref{eq:kN}.  Wagemakers also proves that every such parameter admits infinitely many positive solutions in the corresponding admissible Pell class~\cite[Thm.~2.1 and Sec.~4]{Wagemakers2019}; this also follows from Proposition~\ref{prop:Pell-persistence} once one endpoint seed is known.

\subsection{Point-dependent monic normalization}

The leading coefficient $2$ in \eqref{eq:F-def} is not a unit in $\Z$, so the clean four-state cycle of Section~\ref{sec:four-jump} does not apply directly.  Theorem~\ref{thm:local-reciprocal} still gives an individual transformed divisibility state whenever the quotient coordinate is nonzero, but its coefficient data no longer return by ordinary reciprocity.  At a fixed solution $(m,n)$ one can instead normalize the quadratic modulo both variables, recovering an ordinary monic Vieta jump and making the return obstruction explicit.

Since $mn$ is odd, define, for any $j\in\Z$,
\begin{equation}\label{eq:a-j}
  a_j\defeq\frac{mn+1}{2}+jmn.
\end{equation}
Then
\begin{equation}\label{eq:2a-cong}
  2a_j\equiv1\pmod m,
  \qquad
  2a_j\equiv1\pmod n.
\end{equation}
For
\begin{equation}\label{eq:Qaj}
  Q_{a_j}(T)\defeq T^2-T+a_j,
\end{equation}
we have
\[
  F(T)=2Q_{a_j}(T)+(1-2a_j).
\]
Hence \eqref{eq:two-div} is equivalent, at the current pair $(m,n)$, to
\begin{equation}\label{eq:normalized-pair}
  m\mid Q_{a_j}(n),
  \qquad
  n\mid Q_{a_j}(m).
\end{equation}
We can therefore apply the ordinary Vieta companion-root identity for the normalized monic quadratic; when the companion coordinate is nonzero, it gives a divisibility-state jump.

\begin{proposition}[one locally normalized jump]\label{prop:local-jump}
Let $(m,n)$ satisfy \eqref{eq:two-div}, choose $a=a_j$ as in \eqref{eq:a-j}, and define
\begin{equation}\label{eq:mtilde}
  \widetilde m\defeq\frac{Q_a(n)}{m}.
\end{equation}
Assume $\widetilde m\ne0$.  Then
\begin{equation}\label{eq:Qa-after-jump}
  \widetilde m\mid Q_a(n),
  \qquad
  n\mid Q_a(\widetilde m).
\end{equation}
Moreover,
\begin{equation}\label{eq:return-criterion}
  \widetilde m\mid F(n)
  \quad\Longleftrightarrow\quad
  2a\equiv1\pmod{\widetilde m}.
\end{equation}
Whenever the equivalent conditions in \eqref{eq:return-criterion} hold, the jumped pair $(\widetilde m,n)$ again satisfies the original system \eqref{eq:two-div}.
\end{proposition}

\begin{proof}
By \eqref{eq:mn-coprime} and \eqref{eq:normalized-pair}, Proposition~\ref{prop:quadratic-vieta-bridge} applies to $Q_a$ and gives \eqref{eq:Qa-after-jump}.  Moreover
\[
  F(n)=2Q_a(n)+(1-2a)=2m\widetilde m+(1-2a),
\]
so $\widetilde m\mid F(n)$ is equivalent to $2a\equiv1\pmod{\widetilde m}$.  Under this condition, together with \eqref{eq:2a-cong}, the two $Q_a$-divisibilities are equivalent to the two original $F$-divisibilities.
\end{proof}

\begin{corollary}[finite divisor criterion]\label{cor:divisor-criterion}
Let $(m,n)$ satisfy \eqref{eq:two-div}, and put
\begin{equation}\label{eq:c-quotient}
  c\defeq\frac{F(n)}{m}.
\end{equation}
As $j$ varies in \eqref{eq:a-j}, the candidate coordinate \eqref{eq:mtilde} runs through the arithmetic progression
\begin{equation}\label{eq:mtilde-progression}
  \widetilde m_j
  =\frac{c+n}{2}+jn,
  \qquad
  \widetilde m_j\equiv1-m\pmod n.
\end{equation}
The jumps that return immediately to the original system \eqref{eq:two-div} are therefore exactly the nonzero divisors $d$ of $F(n)$ satisfying
\begin{equation}\label{eq:divisor-class}
  d\equiv1-m\pmod n.
\end{equation}
\end{corollary}

\begin{proof}
Since $F(n)=mc$ and $2a_j=mn+1+2jmn$,
\[
  2Q_{a_j}(n)
  =F(n)-1+2a_j
  =m(c+n+2jn),
\]
which gives the first formula in \eqref{eq:mtilde-progression}.  Moreover $mc=F(n)\equiv1\pmod n$, while $F(m)\equiv0\pmod n$ gives
\[
  m^{-1}\equiv2-2m\pmod n.
\]
Hence $c\equiv2(1-m)\pmod n$, and \eqref{eq:mtilde-progression} follows because $n$ is odd.  Proposition~\ref{prop:local-jump} says that such a jump returns to \eqref{eq:two-div} exactly when $\widetilde m_j\mid F(n)$.  Conversely, every divisor $d\mid F(n)$ in the residue class \eqref{eq:divisor-class} has the form \eqref{eq:mtilde-progression} for a unique $j\in\Z$.
\end{proof}

\begin{remark}[what the local normalization does not imply]\label{rem:local-obstruction}
Proposition~\ref{prop:local-jump} isolates the obstruction to a naive nonunit extension.  The Chinese-remainder choice \eqref{eq:a-j} embeds the current solution $(m,n)$ into the monic system \eqref{eq:normalized-pair}.  The companion-root identity is always valid there, and whenever the new coordinate $\widetilde m$ is nonzero it defines the corresponding divisibility-state jump.  It does \emph{not} follow that the jumped pair satisfies the original $F$-divisibilities: $\widetilde m$ must additionally satisfy \eqref{eq:return-criterion}.  Corollary~\ref{cor:divisor-criterion} shows that, for a fixed current pair, the possible immediate returns form a finite divisor set and this set may be empty.  Consequently the CRT normalization by itself proves neither the existence of a nontrivial return, nor an infinite jump path, nor generation of the full Pell orbit.  Any such conclusion requires an additional global argument, for example a recurrent normalization cycle or a proof that a prescribed sequence of Pell solutions meets the return criterion.
\end{remark}

\begin{example}[no one-step CRT return]\label{ex:no-one-step-return}
The first UVDose--Wagemakers seed is $(m,n)=(17,109)$, for which \eqref{eq:k-quadratic} has $k=13$.  Proposition~\ref{prop:Pell-persistence} therefore places this seed in an infinite positive Pell family.  Nevertheless Corollary~\ref{cor:divisor-criterion} gives no immediate return from this seed.  Indeed,
\[
  F(109)=23545=5\cdot17\cdot277,
  \qquad
  1-17\equiv93\pmod{109},
\]
and none of the positive or negative divisors of $23545$ is congruent to $93$ modulo $109$.  Thus, as $j$ varies in \eqref{eq:a-j}, every monic jump \eqref{eq:mtilde} leaves the original nonunit system after the first step.  This does not rule out a return after several auxiliary states, but it shows concretely that the local CRT jump and the Pell mechanism for infinitude are logically distinct.  Proposition~\ref{prop:uvdose-ladder} supplies the complementary positive phenomenon: its $N=1$ auxiliary-state ladder reaches the same endpoint $(17,109)$ from the trivial pair $(1,1)$, but it does so through the prescribed changing states $P_1,P_0,F$, not by an immediate CRT-normalized return starting from $(17,109)$.
\end{example}

There is also a useful fixed-conic interpretation.  Combining \eqref{eq:k-quadratic} and \eqref{eq:a-j} gives
\begin{equation}\label{eq:normalized-conic}
  m^2+n^2-m-n+a_j=C_jmn,
\end{equation}
where
\begin{equation}\label{eq:Cj}
  C_j\defeq\frac{k+1}{2}+j.
\end{equation}
For a fixed $j$, Proposition~\ref{prop:quadratic-vieta-bridge} applies to \eqref{eq:normalized-conic}.  In particular, the normalized Vieta jump is linear:
\begin{equation}\label{eq:linear-Q-jump}
  \widetilde m=C_jn+1-m.
\end{equation}
Keeping $j$ fixed therefore produces ordinary equation jumping on one conic; returning to the original nonunit divisibility system requires the extra congruence \eqref{eq:return-criterion}.

\section{Auxiliary relation lattices and finite classification}\label{sec:relation-lattices}\label{subsec:relation-lattices}

The preceding sections use quotient dynamics, descent, or recurrence.  There is also a complementary finite method: auxiliary congruences can sometimes give a complete classification without producing a closed arithmetic orbit.  The underlying linear-algebraic statement is independent of the particular polynomial family.

Fix a nonzero integer modulus $N$ and an evaluation vector
\[
  \mathbf m=(m_1,\ldots,m_{d-1},1)\in\Z^d.
\]
Define the relation lattice
\begin{equation}\label{eq:relation-lattice}
  \Lcal_N(\mathbf m)
  \defeq
  \{\mathbf v\in\Z^d:\ \mathbf v\cdot\mathbf m\equiv0\pmod N\}.
\end{equation}
The final coordinate $1$ is the constant monomial.  For example, a
relation
\[
  aX^5+bX+cY^5+d\equiv0\pmod{XY}
\]
corresponds to $N=XY$ and
$\mathbf m=(X^5,X,Y^5,1)$.

\begin{proposition}[index and determinant test]\label{prop:lattice-index}
The lattice \eqref{eq:relation-lattice} has index
\[
  [\Z^d:\Lcal_N(\mathbf m)]=|N|.
\]
A basis is
\[
  e_i-m_i e_d\quad(1\le i<d),
  \qquad
  N e_d.
\]
If $d$ linearly independent relation vectors are the rows of an integer
matrix $V$ and generate a sublattice $L_0\subseteq\Lcal_N(\mathbf m)$,
then
\begin{equation}\label{eq:relation-index}
  [\Lcal_N(\mathbf m):L_0]
  =\frac{|\det V|}{|N|}.
\end{equation}
In particular, $N\mid\det V$.
\end{proposition}

\begin{proof}
The map
\[
  \varphi:\Z^d\longrightarrow\Z/N\Z,
  \qquad
  \mathbf v\longmapsto\mathbf v\cdot\mathbf m\pmod N
\]
is surjective because the last coordinate of $\mathbf m$ equals $1$.
Its kernel is \eqref{eq:relation-lattice}, so the first assertion follows.
The displayed vectors plainly lie in the kernel, and their determinant
has absolute value $|N|$, hence they form a basis.  Finally,
$[\Z^d:L_0]=|\det V|$ and multiplicativity of lattice indices gives
\eqref{eq:relation-index}.
\end{proof}

\begin{corollary}[full-rank finite classification certificate]
\label{cor:finite-classification}
Let $\mathcal S$ be a set of nonzero integer solutions $(X,Y)$ of a fixed
Diophantine problem.  Suppose that every $(X,Y)\in\mathcal S$ satisfies
$d$ fixed congruences
\[
  \mathbf v_i\cdot\mathbf m(X,Y)\equiv0\pmod{XY},
  \qquad 1\le i\le d,
\]
where the vectors $\mathbf v_i\in\Z^d$ are independent of $(X,Y)$ and
the last coordinate of $\mathbf m(X,Y)$ is $1$.  Let $V$ be the integer
matrix whose rows are the $\mathbf v_i$, and assume
\[
  D\defeq\det V\ne0.
\]
Then every $(X,Y)\in\mathcal S$ satisfies
\begin{equation}\label{eq:XY-divides-D}
  XY\mid D.
\end{equation}
Consequently all nonzero solutions can be found by the finite procedure:
factor $D$, enumerate its signed divisors $N$, enumerate the ordered
integer factor pairs $XY=N$, and retain exactly those pairs satisfying
the original Diophantine conditions.

In particular, for four relations
\[
  a_iX^5+b_iX+c_iY^5+d_i\equiv0\pmod{XY},
  \qquad 1\le i\le4,
\]
with fixed coefficient matrix
\[
  V=
  \begin{pmatrix}
    a_1&b_1&c_1&d_1\\
    a_2&b_2&c_2&d_2\\
    a_3&b_3&c_3&d_3\\
    a_4&b_4&c_4&d_4
  \end{pmatrix},
\]
nonzero determinant gives a complete finite classification certificate
for the nonzero solutions.
\end{corollary}

\begin{proof}
For a solution $(X,Y)$ set $N=XY$ and
\[
  \mathbf m(X,Y)=(X^5,X,Y^5,1)^{\mathsf T}
\]
in the displayed four-variable model, or use the corresponding evaluation
vector in the general case.  Proposition~\ref{prop:lattice-index} gives
$N\mid D$.  Equivalently, from
$V\mathbf m\equiv0\pmod N$, multiplication by
$\operatorname{adj}(V)$ gives
\[
  D\mathbf m\equiv0\pmod N,
\]
and the last coordinate of $\mathbf m$ is $1$, so $N\mid D$.
There are only finitely many signed divisors $N$ of $D$ and finitely many
ordered integer factorizations $N=XY$.  Every solution occurs in this
list, and direct substitution into the original system discards all false
candidates, proving completeness of the procedure.
\end{proof}

\begin{remark}[classification, not merely boundedness]
\label{rem:classification-not-bound}
The conclusion $XY\mid D$ is stronger than $|XY|\le|D|$.  It reduces the
original Diophantine problem to a finite divisor search.  Smith normal
form, congruence conditions, signs, positivity, or other information from
the original system may prune that list further, but they are not needed
for completeness.  Conversely, if the original problem has infinitely
many nonzero solutions, then no fixed full-rank family of auxiliary
relations can be valid for all of them: either the fixed relation matrix
is singular or the relations themselves must vary with the solution.
\end{remark}

\begin{remark}[Smith normal form]\label{rem:snf}
For a target relation vector $\mathbf t$, membership in the lattice
spanned by known relations is an ordinary Smith-normal-form problem.
If
\[
  U V^{\mathsf T} W=\operatorname{diag}(s_1,\ldots,s_d)
\]
is a Smith normal form, then $V^{\mathsf T}\lambda\equiv\mathbf t\pmod N$
is solvable precisely when
\[
  \gcd(s_j,N)\mid (U\mathbf t)_j
\]
for every $j$.  Thus the relevant invariant is not merely the number of
available surfaces, but the index of the sublattice they generate.
\end{remark}

\section{Selected applications}\label{sec:applications}

We collect here several short applications illustrating how the preceding framework can lead to descent, explicit parametrization, or classification of solution sets.  They are logically independent of the detailed nonunit case study in Section~\ref{sec:nonunit}.

\subsection{Self-reciprocal chains and OEIS A140362}
\begin{example}[Mills's quadratic chains]\label{ex:mills-chain}
For $P_a(T)\defeq T^2+aT+1$ one has $P_a^{\vee}=P_a$, so quotient jumping stays in the same state and
\[
  u_{j-1}u_{j+1}=u_j^2+au_j+1.
\]
Mills showed that these chains satisfy an affine second-order recurrence~\cite{Mills1953}.  For $a=0$ and $a=1$ the positive states each form a single jump orbit.  By symmetry assume $1\le x\le y$.  If $x=y$, then $x\mid P_a(x)$ gives $x\mid1$, so $x=1$.  If $1<x<y$ and $y\mid P_a(x)$, then
$y'=P_a(x)/y$ is positive and $y'<x$; for $a=1$ the only borderline case $y=x+1$ is excluded by $P_1(x)\equiv1\pmod{x+1}$.  Descent therefore reaches $x=1$, giving terminal pairs $(1,1),(1,2)$ for $a=0$ and $(1,1),(1,3)$ for $a=1$.

Thus the $a=0$ chain is
\[
  1,1,2,5,13,34,89,233,\ldots,
  \qquad u_{j+1}=3u_j-u_{j-1},
\]
which is OEIS A001519~\cite{OEIS}; apart from the repeated initial $1$, it consists of odd-indexed Fibonacci numbers.  Write
\[
  \sigma_2(N)\defeq\sum_{d\mid N}d^2.
\]
OEIS A140362 consists of semiprimes $N$ for which $N\mid\sigma_2(N)$~\cite{OEIS}.  A square semiprime is impossible modulo $p^2$; for $N=pq$ with distinct primes,
\[
  pq\mid\sigma_2(pq)
  \quad\Longleftrightarrow\quad
  p\mid q^2+1,\quad q\mid p^2+1.
\]
Hence A140362 is exactly the set of products of consecutive prime terms on this chain, beginning
$10,65,20737,\ldots$.  For $a=1$ the positive chain is OEIS A101368,
$1,1,3,13,61,291,\ldots$; allowing signs produces additional components, illustrating that reversibility need not imply connectedness.
\end{example}

\subsection{Higher-degree three-jump examples}
Proposition~\ref{prop:three-jump-power} has a useful fixed-state sibling.  For
$\cS(T+1,(T+1)^d)$ both exchange polynomials are $1$-self-reciprocal, so the three coordinate jumps form the involution
\[
  q\defeq\frac{(z+1)^d}{t},\qquad
  z'\defeq\frac{q+1}{z},\qquad
  t'\defeq\frac{(z'+1)^d}{q}.
\]
At a positive fixed point one has $z'=z$, hence $q=z^2-1$ and
\[
  t=\frac{(z+1)^d}{z^2-1}
   =\frac{(z+1)^{d-1}}{z-1}.
\]
Since $z+1\equiv2\pmod{z-1}$, integrality is equivalent to $z-1\mid2^{d-1}$.  For $d=4$, OEIS A350916 consists of the positive integers $k$ for which $(k+1)^4$ has a divisor congruent to $-1\pmod k$; it is the union of ten second-order recurrence families $(b_n)$ satisfying
\[
  (b_n+1)^4=(b_nb_{n-1}-1)(b_nb_{n+1}-1).
\]
For each such family, $t_n\defeq b_nb_{n-1}-1$ gives
$(t_n,b_n)\in\cS(T+1,(T+1)^4)$~\cite{OEIS}.  Multiplying the second polynomial by a nonunit immediately changes the picture: for $2(T+1)^3$ one has $\Rcal_2(T+1)=T+2$, so a fixed reciprocal state becomes a changing one after a single jump.

For $d=3$, the return in Proposition~\ref{prop:three-jump-power} recovers a classical classification of N.~Osipov, problem M1787 in \emph{Kvant} No.~5 (2001), with a solution published in No.~2 (2002)~\cite{OsipovM1787Problem2001,OsipovM1787Solution2002}.
\begin{corollary}[Osipov's cubic classification]\label{cor:cubic-minus-one-classification}
For $p,q>1$,
\[
  p\mid q^3-1,\qquad q\mid p-1
\]
holds if and only if either
\[
  p=q^2+q+1,
\]
or $q=s^2$ and $p=s^3+1$ for some integer $s>1$.
\end{corollary}
\begin{proof}
Use the notation of Proposition~\ref{prop:three-jump-power}.  Since $p>q$,
$u=(q^3-1)/p<q^2$, and $u=kq-1$ gives $k\le q$.  Equality $k=q$ would imply
\[
  p=\frac{q^3-1}{q^2-1}
   =q+\frac1{q+1},
\]
which is impossible for integral $p$.  Hence $k<q$.  If the returned point $(p',k)$ had $p',k>1$, applying the same inequality to it would give $q<k$, a contradiction.  Hence $k=1$ or $p'=1$.  These alternatives give respectively
$p=q^2+q+1$ and $q=k^2$, $p=k^3+1$.  The converses are immediate by factorization.
\end{proof}

\subsection{The parabolic quadratic case}
The exceptional value $H=2$ in Proposition~\ref{prop:quadratic-vieta-bridge} has a simple concrete realization.  If
\[
  \frac{2m^2+2n^2+1}{(2m+1)(2n+1)}\in\Z,
\]
put $x\defeq2m+1$, $y\defeq2n+1$, and $Q(T)\defeq T^2-2T+4$.  Then
\[
  x^2+y^2-2x-2y+4=Hxy,
\]
with $H$ even, while $\Rcal_4Q=4Q$.  All coordinates are odd, so the scalar $4$ causes no normalization obstruction.  For $x>y\ge3$, the companion root
\[
  x'=\frac{Q(y)}x
\]
satisfies $0<x'<y$ because $Q(y)=y^2-2y+4<y^2<xy$.  Vieta descent therefore reaches $(1,1)$ or $(1,3)$, both with $H=2$, and hence the original quotient equals $1$.  The recurrence
$u_{j+1}=2u_j-u_{j-1}+2$ becomes, after $u_j=2t_j+1$,
\[
  t_{j+1}=2t_j-t_{j-1}+1.
\]
Thus the positive solutions are precisely consecutive triangular numbers
\[
  (m,n)=(T_r,T_{r+1})\ \text{or}\ (T_{r+1},T_r),
  \qquad T_r\defeq\frac{r(r+1)}2,
  \qquad r\ge1.
\]

\subsection{A nonunit cubic ascent}
Grechuk's ascent for
$y^2+7xyz+3x^3-2=0$ begins with
$x\mid y^2-2$ and $y\mid3x^3-2$~\cite{GrechukAscent2024}.  If
$X\defeq(y^2-2)/x$, Theorem~\ref{thm:local-reciprocal} gives
\[
  y\mid\Rcal_{-2}(3T^3-2)(X)=-2(X^3+12).
\]
For odd $y$ the scalar cancels, yielding $Y\defeq(X^3+12)/y$; a second jump gives
$X\mid-2(Y^2-72)$, where Grechuk's parity argument again cancels $2$.  Thus this ascent is a concise higher-degree example in which the raw reciprocal rule remains valid without endpoint-unit hypotheses and exposes exactly where arithmetic cancellation enters.

\section{Conclusion and open problems}\label{sec:outlook}

Congruence jumping is useful only insofar as changing polynomial states can be converted into arithmetic information about the original solution set.  The local reciprocal theorem separates the always-valid quotient step from the genuinely arithmetic problem of primitive normalization.  Companion surfaces give a broad positive lifting mechanism, while the four-step gauge identity explains why short coefficient recurrences can coexist with much subtler pointwise integrality.

The strongest rigidity appears in degree two.  Finite integral chains already force high powers of the denominator of the conic parameter into a fixed Cassini-type determinant, and the infinite exceptional case is completely described.  Recurrent nonunit states can nevertheless lift indefinitely: the four-cycle of Proposition~\ref{prop:recurrent-nonunit-cycle} has infinitely many positive integral points, and Proposition~\ref{prop:lucas-pell-cycle} identifies the resulting Lucas--Pell dynamics uniformly in the parameter.  The system in Section~\ref{sec:nonunit} shows at the same time that Pell infinitude and congruence-jump return are distinct arithmetic phenomena.

A different route to classification is furnished by the auxiliary relation lattice of Section~\ref{sec:relation-lattices}.  Full rank forces $XY\mid\det V$ and hence turns a global Diophantine problem into a finite divisor search.  This criterion is elementary but useful precisely because it does not require a closed jump orbit.

The remaining questions are therefore solution-theoretic.
\begin{enumerate}[label=\textup{(\arabic*)}]
\item \textbf{Integral lifting beyond companion surfaces.}  For changing primitive states not covered by the companion-surface mechanism, characterize the integer seeds whose successive normalizations remain valid.  Can the lifting locus be described by finitely many local conditions at the obstruction primes?
\item \textbf{Recurrent nonunit components.}  Which recurrent nonunit state cycles admit infinite arithmetic lifts, and how do their reversible jump graphs decompose into components that can be parametrized or descended explicitly?
\item \textbf{Finite classification certificates.}  For concrete polynomial divisibility systems, when do natural auxiliary constructions generate a fixed full-rank relation lattice?  In the rank-deficient case, what arithmetic information can replace the determinant certificate?
\end{enumerate}

These problems keep the endpoint explicit: determine which integer solutions occur, how they are organized, and when the organization can be converted into descent, recurrence or Pell parametrization, or a finite exhaustive search.

\section*{Data availability}
No data was used for the research described in the article.

\section*{Declaration of generative AI and AI-assisted technologies in the manuscript preparation process}
During the development and preparation of this work, the authors used OpenAI ChatGPT (GPT-5.6 Sol) to assist with exploratory calculations and checks of proof steps and examples, and with manuscript organization, language, and readability.  After using this tool, the authors reviewed and edited the content as needed and take full responsibility for the content of the publication.

\printbibliography

@article{Mills1953,
  author  = {Mills, W. H.},
  title   = {A system of quadratic Diophantine equations},
  journal = {Pacific Journal of Mathematics},
  volume  = {3},
  number  = {1},
  pages   = {209--220},
  year    = {1953},
  doi     = {10.2140/pjm.1953.3.209}
}

@article{Mills1954,
  author  = {Mills, W. H.},
  title   = {A method for solving certain Diophantine equations},
  journal = {Proceedings of the American Mathematical Society},
  volume  = {5},
  number  = {3},
  pages   = {473--475},
  year    = {1954},
  doi     = {10.1090/S0002-9939-1954-0062757-3}
}

@article{Mills1956,
  author  = {Mills, W. H.},
  title   = {Certain Diophantine equations linear in one unknown},
  journal = {Canadian Journal of Mathematics},
  volume  = {8},
  pages   = {5--12},
  year    = {1956},
  doi     = {10.4153/CJM-1956-003-2}
}

@book{Grechuk2024,
  author    = {Grechuk, Bogdan},
  title     = {Polynomial Diophantine Equations: A Systematic Approach},
  publisher = {Springer},
  location  = {Cham},
  year      = {2024},
  doi       = {10.1007/978-3-031-62949-5}
}

@online{Lemmermeyer2026,
  author       = {Lemmermeyer, Franz},
  title        = {Vieta jumping and small norms in quadratic number fields},
  year         = {2026},
  eprint       = {2601.15229},
  eprinttype   = {arXiv},
  eprintclass  = {math.NT}
}

@article{Mohanty1977,
  author  = {Mohanty, S. P.},
  title   = {A system of cubic Diophantine equations},
  journal = {Journal of Number Theory},
  volume  = {9},
  number  = {2},
  pages   = {153--159},
  year    = {1977},
  doi     = {10.1016/0022-314X(77)90017-8}
}

@article{Dofs1993,
  author  = {Dofs, E.},
  title   = {On extensions of 1 chains},
  journal = {Acta Arithmetica},
  volume  = {65},
  number  = {3},
  pages   = {249--258},
  year    = {1993},
  doi     = {10.4064/aa-65-3-249-258}
}

@incollection{Bier2014,
  author    = {Bier, Thomas},
  title     = {Finite or infinite number of solutions of polynomial congruences in two positive integer variables},
  booktitle = {Automorphic Forms: Research in Number Theory from Oman},
  editor    = {Heim, Bernhard and Al-Baali, Mohammed and Ibukiyama, Tomoyoshi and Rupp, Florian},
  series    = {Springer Proceedings in Mathematics \& Statistics},
  volume    = {115},
  publisher = {Springer International Publishing},
  location  = {Cham},
  pages     = {11--26},
  year      = {2014},
  doi       = {10.1007/978-3-319-11352-4_2}
}

@article{Jacobsthal1939,
  author  = {Jacobsthal, Ernst},
  title   = {Zahlentheoretische Eigenschaften ganzzahliger Polynome},
  journal = {Compositio Mathematica},
  volume  = {6},
  pages   = {407--427},
  year    = {1939},
  url     = {https://www.numdam.org/item/CM_1939__6__407_0/},
  urldate = {2026-09-27}
}

@article{Mordell1952,
  author  = {Mordell, L. J.},
  title   = {The congruence $ax^3+by^3+c\equiv0\pmod{xy}$, and integer solutions of cubic equations in three variables},
  journal = {Acta Mathematica},
  volume  = {88},
  number  = {1},
  pages   = {77--83},
  year    = {1952}
}

@article{Schinzel2015,
  author  = {Schinzel, Andrzej},
  title   = {On the congruence $f(x)+g(y)+c\equiv0\pmod{xy}$ (completion of {Mordell}'s proof)},
  journal = {Acta Arithmetica},
  volume  = {167},
  number  = {4},
  pages   = {347--374},
  year    = {2015},
  doi     = {10.4064/aa167-4-3}
}

@article{Schinzel2018,
  author  = {Schinzel, Andrzej},
  title   = {On the congruence $f(x)+g(y)+c\equiv0\pmod{xy}$, {II} (the quadratic case)},
  journal = {Acta Arithmetica},
  volume  = {184},
  number  = {1},
  pages   = {1--6},
  year    = {2018},
  doi     = {10.4064/aa8449-7-2016}
}

@article{TengelyUlas2022,
  author  = {Tengely, Szabolcs and Ulas, Maciej},
  title   = {Diophantine problems related to cyclic cubic and quartic fields},
  journal = {Journal of Number Theory},
  volume  = {240},
  pages   = {656--684},
  year    = {2022},
  doi     = {10.1016/j.jnt.2022.01.003}
}

@article{KollarVillalobosPaz2025,
  author  = {Koll{\'a}r, J{\'a}nos and Villalobos-Paz, David},
  title   = {Cubic surfaces with infinite, discrete automorphism group},
  journal = {Rendiconti Lincei. Matematica e Applicazioni},
  volume  = {36},
  number  = {3},
  pages   = {563--590},
  year    = {2025},
  doi     = {10.4171/RLM/1083},
  eprint       = {2410.03934},
  eprinttype   = {arXiv},
  eprintclass  = {math.AG}
}

@online{KollarLi2024,
  author       = {Koll{\'a}r, J{\'a}nos and Li, J{\'a}nos},
  title        = {The {Mordell--Schinzel} conjecture for cubic Diophantine equations},
  year         = {2024},
  eprint       = {2412.12080},
  eprinttype   = {arXiv},
  eprintclass  = {math.NT}
}

@article{Alperin2011,
  author  = {Alperin, Roger C.},
  title   = {Integer sequences generated by $x_{n+1}=(x_n^2+A)/x_{n-1}$},
  journal = {The Fibonacci Quarterly},
  volume  = {49},
  number  = {4},
  pages   = {362--365},
  year    = {2011},
  doi     = {10.1080/00150517.2011.12428037}
}

@book{BerstelReutenauer2011,
  author    = {Berstel, Jean and Reutenauer, Christophe},
  title     = {Noncommutative Rational Series with Applications},
  series    = {Encyclopedia of Mathematics and its Applications},
  volume    = {137},
  publisher = {Cambridge University Press},
  year      = {2011},
  doi       = {10.1017/CBO9780511760860}
}

@article{LamPylyavskyy2016,
  author  = {Lam, Thomas and Pylyavskyy, Pavlo},
  title   = {Laurent phenomenon algebras},
  journal = {Cambridge Journal of Mathematics},
  volume  = {4},
  number  = {1},
  pages   = {121--162},
  year    = {2016},
  doi     = {10.4310/CJM.2016.v4.n1.a2}
}

@online{OEIS,
  author       = {{OEIS Foundation Inc.}},
  title        = {The On-Line Encyclopedia of Integer Sequences},
  url          = {https://oeis.org},
  urldate      = {2026-09-28}
}

@unpublished{UVDose2016,
  author = {{UVDose}},
  title  = {Double divisibilit{\'e}},
  note   = {Unpublished two-page note attached to the 2016 les-mathematiques.net discussion; author cited by public username},
  date   = {2016-03}
}

@online{Wagemakers2019,
  author  = {Wagemakers, Gerard},
  title   = {Solving simultaneous $n+1\mid m^2+1$ and $m+1\mid n^2+1$},
  year    = {2019},
  note    = {Unpublished research note; PDF metadata dated July 7, 2019; archived copy},
  url     = {https://web.archive.org/web/20210506154406if_/http://www.effectiveit.nl/math/DL_2SD_MSE.php},
  urldate = {2026-09-27}
}

@online{OsipovM1787Problem2001,
  author  = {Osipov, N.},
  title   = {Problem {M1787}},
  organization = {Kvant},
  number  = {5},
  pages   = {17},
  year    = {2001},
  url     = {https://kvant.mccme.ru/pdf/2001/05/kv0501probl.pdf},
  urldate = {2026-09-27}
}

@online{OsipovM1787Solution2002,
  author  = {Osipov, N.},
  title   = {Solution to problem {M1787}},
  organization = {Kvant},
  number  = {2},
  pages   = {14},
  year    = {2002},
  url     = {https://kvant.mccme.ru/pdf/2002/02/kv0202solutions.pdf},
  urldate = {2026-09-27}
}

@online{GrechukAscent2024,
  author  = {Grechuk, Bogdan},
  title   = {Answer to ``The new shortest open cubic equations''},
  organization = {MathOverflow},
  url     = {https://mathoverflow.net/q/466803},
  urldate = {2026-09-28}
}

@online{GrechukShatrov2024,
  author       = {Grechuk, Bogdan and Shatrov, Denis},
  title        = {Discussion on {MathOverflow}: Can $9xy$ divide $1+x^2+x^3+y^2$?},
  organization = {MathOverflow},
  date         = {2024-03},
  url          = {https://mathoverflow.net/questions/466362/can-9xy-divide-1x2x3y2},
  urldate      = {2026-09-30}
}

@online{GrechukAgbanwa2026,
  author       = {Grechuk, Bogdan and Agbanwa, Jamal},
  title        = {On the polynomial values represented by binary quadratic forms},
  date         = {2026-09-15},
  eprint       = {2607.06627},
  eprinttype   = {arXiv},
  eprintclass  = {math.GM},
  doi          = {10.48550/arXiv.2607.06627},
  note         = {Version 2}
}
\end{document}